\documentclass[12pt]{article}
\usepackage{etex}
\usepackage{amsfonts,amsmath,amsthm,amssymb,hyperref,fancyhdr,float,caption,mathtools}
\usepackage[all,pdf,cmtip]{xy}
\usepackage{caption}
\usepackage{verbatim}
\usepackage{amsthm}
\usepackage{tikz}
\usepackage{tikz-3dplot}
\tikzset{
  pics/carc/.style args={#1:#2:#3}{
    code={
      \draw[pic actions] (#1:#3) arc(#1:#2:#3);
    }
  }
}

\usetikzlibrary{calc,fadings,decorations.pathreplacing}

\usetikzlibrary{patterns}
\theoremstyle{theorem}
\newtheorem{defn}{Definition}[section]
\newtheorem*{theorem*}{Theorem}
\theoremstyle{definition}
\newtheorem{remark}[defn]{Remark}
\newtheorem{example}[defn]{Example}

\newtheoremstyle{parag}
  {\topsep}   
  {\topsep}   
  {}  
  {}       
  {\bfseries} 
  {}         
  { } 
  {}          
\theoremstyle{parag}
\newtheorem{pp}[defn]{}
\newtheoremstyle{rem}
  {\topsep}   
  {\topsep}   
  {\it}  
  {}       
  {\bfseries} 
  {}         
  { } 
  {}          
\theoremstyle{rem}

\theoremstyle{plain}
\newtheorem{prop}[defn]{Proposition}
\newtheorem{lemma}[defn]{Lemma}
\newtheorem{theorem}[defn]{Theorem}
\newtheorem{corollary}[defn]{Corollary}

\newcommand{\nc}{\newcommand}

\nc{\sU}{\mathcal{U}}
\nc{\sC}{\mathcal{C}}
\nc{\D}{\mathcal{D}}
\nc{\sF}{\mathsf F}
\nc{\sV}{\mathsf V}
\nc{\sS}{\mathcal{S}}
\nc{\sv}{\mathsf{v}}
\nc{\sw}{\mathsf{w}}
\nc{\CC}{\mathbb C}
\nc{\RR}{\mathbb R}
\nc{\NN}{\mathbb N} 
\nc{\ZZ}{\mathbb Z} 
\nc{\HH}{\mathbb H}
\nc{\PP}{\mathbb P} 
\nc{\OO}{\mathcal{O}} 
\nc{\KK}{\mathcal{K}}
\nc{\II}{\mathcal{I}}
\nc{\del}{\partial}
\nc{\delbar}{\overline{\partial}}
\nc{\isom}{\cong}
\nc{\E}{\mathbb{E}}
\nc{\e}{\mathbf{e}}
\nc{\Gg}{\mathfrak{g}}

\nc{\Tt}{\mathfrak{t}}
\nc{\bb}[1]{\mathbb{#1}}
\nc{\tensor}{\otimes}
\nc{\ul}{\underline}
\nc{\ol}{\overline}
\newcommand{\mc}{\DeclareMathOperator}
\mc{\End}{\mathrm{End}}
\mc{\Hom}{\mathrm{Hom}}
\mc{\Ann}{\mathrm{Ann}}
\mc{\Aut}{\mathrm{Aut}}
\mc{\Pic}{\mathrm{Pic}}
\mc{\Sym}{\mathrm{Sym}}
\mc{\rk}{\mathrm{rk\,}}
\mc{\so}{\mathfrak{so}}
\mc{\Dol}{\mathrm{Dol}}
\mc{\hol}{\mathrm{hol}}
\mc{\tot}{\mathrm{tot}}
\mc{\MM}{\mathcal{M}}
\mc{\X}{\mathcal{X}}
\mc{\Y}{\mathcal{Y}}
\mc{\Ext}{\mathrm{Ext}}

\mc{\At}{\mathit{At}}
\mc{\Res}{\mathrm{Res}}
\nc{\IP}[1]{\!\left<{#1}\right>}
\nc{\sL}{\mathcal{L}}
\nc{\Tr}{\mathrm{Tr}}

\nc{\compressed}{compressed}

\nc{\id}{\matheuler{I}}

\newcommand{\blowup}[2]{\mathsf{Bl}_{#2}(#1)}

\usepackage{geometry} 
\newlength{\leftside}
\newlength{\rightside}
\usepackage{lmodern}

\usepackage{sectsty} 
\sectionfont{\bfseries\upshape\large}
\subsectionfont{\bfseries\upshape\normalsize} 
\subsubsectionfont{\bfseries\upshape\small}

\usepackage{color}
\definecolor{lightgr}{rgb}{.8,.8,.8}
\definecolor{urlcolor}{rgb}{.1,.1,.4}
\definecolor{blackblue}{rgb}{.1,.1,.3}
\definecolor{tocolor}{rgb}{.1,.1,.5}
\definecolor{urlcolor}{rgb}{.2,.2,.6}
\definecolor{linkcolor}{rgb}{.1,.1,.6}
\definecolor{citecolor}{rgb}{.6,.2,.1}
\definecolor{remcolor}{rgb}{.6,.2,.2}
\definecolor{blue}{rgb}{0,0,.6}

\hypersetup{colorlinks=true, urlcolor=urlcolor, linkcolor=linkcolor, citecolor=citecolor}

\begingroup
\hypersetup{linkcolor=blue}

\endgroup

\usepackage{lastpage}

\date{}

\author{
\noindent
  Marco Gualtieri 
  \quad
  Songhao Li
  \quad
  \'Alvaro Pelayo
  \quad
  Tudor S. Ratiu
}

\newcommand{\addresses}{
	\bigskip

	\noindent Marco Gualtieri \qquad \texttt{mgualt@math.toronto.edu}\\
	\footnotesize{\textsc{
	Department of Mathematics, University of Toronto \\
	Bahen Center, 40 St. George Street Room 6290 \\
	Toronto, Ontario, M5S 2E4 Canada}}
	
	\bigskip

	\noindent Songhao Li \qquad \texttt{sli@math.wustl.edu}\\
	\footnotesize{\textsc{
	Department of Mathematics, Washington University in Saint Louis \\
	1 Brookings Drive \\
	St. Louis, MO 63130-4899, USA}}

	\bigskip

	\noindent \'Alvaro Pelayo \qquad \texttt{alpelayo@math.ucsd.edu}\\
	\footnotesize{\textsc{
	Department of Mathematics, University of California, San Diego \\
	9500 Gilman Drive\\
	La Jolla, CA 92093-0112, USA}}

	\bigskip

	\noindent Tudor S. Ratiu \qquad \texttt{tudor.ratiu@epfl.ch}\\
	\footnotesize{\textsc{
	Department of Mathematics, Jiao Tong University \\
	800 Dongchuan Road\\
	Minhang, Shanghai, 200240 China \\
	and \\
	Section de Math\'ematiques, 
	\'Ecole Polytechnique F\'ed\'erale de Lausanne\\
	Station 8, CH-1015 Lausanne, Switzerland}}
}

\begin{document}
\title{
\vspace{-3em}
\sffamily{The tropical momentum map:\\ a classification of toric log symplectic manifolds}}

\maketitle
\renewcommand{\abstractname}{\vspace{-8ex}} 
\abstract{
  We give a generalization of toric symplectic geometry to 
  Poisson manifolds which are symplectic away from a
  collection of hypersurfaces forming a normal crossing configuration.  
  We introduce the tropical momentum map, which takes values in a 
  generalization of affine space called a log affine manifold.  
  Using this momentum map, we obtain a complete 
  classification of such manifolds in terms of decorated log affine polytopes, hence extending the classification of symplectic toric
manifolds achieved by Atiyah, Guillemin-Sternberg, Kostant, and Delzant.}
\vspace{-3em}\tableofcontents
\vspace{2em}
\vfill
\pagebreak
\section{Introduction}

Toric symplectic geometry was revolutionized in the 1970s and 1980s
by Atiyah, Guillemin-Sternberg, Kostant, and Delzant \cite{MR642416,MR664117,MR0364552,MR984900}, who essentially proved that toric symplectic manifolds are encoded combinatorially by
a rational polytope, which is the image of the classical momentum map associated to the toric action (this is usually referred to as the \emph{Delzant correspondence}). 
In the present
paper, we present a generalization of toric symplectic geometry to 
a class of Poisson manifolds, called \emph{log symplectic manifolds}, which are generically symplectic but degenerate along a normal crossing configuration of smooth hypersurfaces.   
 We give a classification analogous to the one given by the aforementioned
authors, but incorporating new invariants coming from the degeneracy loci.  

\vspace{.3ex}
The study of toric log symplectic manifolds was initiated in the 
paper of Guillemin, Miranda, Pires, and Scott~\cite{Guillemin3}, who developed an extension of the Delzant correspondence in the case where the degeneracy locus  
of the Poisson structure is a smooth hypersurface. 
Degeneracy loci for Poisson structures of interest in Lie theory or in the study of moduli problems are, however, often highly singular (see, for example, \cite{MR3100779,Pym:2015fk}). In this paper, we consider the mildest possible class of singularities: normal crossing hypersurfaces. This class is particularly interesting, as it is simple enough to allow for the presence of toric symmetries, yet is extremely rich in  examples and in combinatorial structure, in comparison to the nonsingular case.
The normal crossing condition on the degeneracy locus is also natural by analogy with the study of algebraic varieties, where Hironaka's theorem implies that any smooth variety may be realized as the complement of a normal crossing divisor in a smooth compactification.

\vspace{.3ex}
To make progress on the normal crossing case, 
it is necessary to reevaluate the nature of the momentum map itself, and to investigate the geometric structure present on its image, which is the direct analog of the rational polytope of Atiyah et al. mentioned above.
We show that a global \emph{tropical momentum map} may be defined, but the geometry and topology of its codomain is significantly more involved than in the case, where the degeneracy locus is smooth (or empty). In the classical case, the momentum image is contractible, while in our case there is no a priori constraint on its topology.   In Example~\ref{gen1}, for instance, we obtain a momentum image diffeomorphic to a surface of genus 1 with  boundary, corresponding to a log symplectic 4-manifold which, interestingly, admits no symplectic structure due to the vanishing of its Seiberg-Witten invariants.

\vspace{.3ex}
In order to use the tropical momentum map effectively to achieve a classification, we introduce several ideas from other fields, such as the Mazzeo-Melrose decomposition for manifolds with corners, the concept of free divisors in algebraic geometry, and the notion of tropicalization from tropical geometry. Indeed, the codomains for tropical momentum maps are assembled
from elementary pieces, which are partial compactifications of
affine spaces closely related to the extended tropicalizations of toric varieties defined by Kajiwara~\cite{MR2428356} and Payne~\cite{MR2511632}. 
In fact, our work was initially motivated by the problem of defining 
momentum maps for Caine's examples of real Poisson structures on complex toric varieties~\cite{MR2859234}. Such a variety may be blown up along its circle
fixed point loci to produce a log symplectic manifold with corners, 
and in this way we may identify its tropical momentum map with its
tropicalization morphism.

\vspace{.3ex}
By carefully selecting these tools, we are able to reach the classification with a minimum of technical fuss.  This is a simplified restatement of our main result, which appears with full generality and detail as Theorem~\ref{delzant} in Section~\ref{sec5}: 

\begin{theorem*}[Classification of Hamiltonian toric log symplectic  manifolds]
There is a one-to-one correspondence between equivariant isomorphism classes of oriented compact connected toric Hamiltonian log symplectic $2n$--manifolds and equivalence classes of pairs $(\Delta, M)$, where $\Delta$ is a compact convex log affine polytope of dimension $n$ satisfying the Delzant condition and $M$ is a principal $n$-torus bundle over $\Delta$ with vanishing obstruction class.
\end{theorem*}

Placing our work in the larger context of the study of integrable systems (see, for example \cite{MR2801777}), this theorem provides a classification of a large family of toric integrable systems in which the base of the system is a manifold whose integral affine structure  is allowed to degenerate in a controlled manner along a stratification. The results of this paper provide further evidence that tools from Lie algebroid theory may be fruitfully exploited to understand and classify the singular behavior of the affine structure on the base of integrable systems, and therefore to further our ability to classify integrable systems at large.

\vspace{.3ex}
The structure of the paper is as follows: In Section~\ref{sec2}, we define log affine manifolds and provide a method 
for constructing a great variety of examples by a welding procedure. In Section~\ref{sec3}, we define toric log symplectic manifolds and classify those with principal torus actions.  In section~\ref{cuts}, we describe the analogue of the Delzant polytope and provide several examples, including the key example~\ref{gen1}.  Finally, in Section~\ref{sec5}, we establish 
the correspondence between log affine polytopes and toric log symplectic manifolds.

\vspace{1em}
\noindent {\it Acknowledgements:}

\noindent We thank Andrew Dancer, Matthias Franz, Eckhard Meinrenken, David Speyer, and Rafael Torres for many insightful remarks. M.G. is partially supported by an NSERC Discovery Grant. A.P. is partially supported by NSF Grant DMS-1055897 and DMS-1518420, the STAMP Program at the ICMAT research institute (Madrid), and ICMAT Severo Ochoa grant Sev-2011-0087. T.S.R. is partially supported by  Swiss National Science Foundation grant NCCR SwissMAP. Finally, we are very grateful to an anonymous referee for 
comments which have improved the paper.

\pagebreak
\section{Notation and conventions}
\begin{enumerate}

\item
{\it Manifolds with corners:} Throughout the paper, manifolds 
have boundary and corners.  We use the definition
in which each local chart has codomain given by an intersection of coordinate half-spaces $x_i \geq 0$, in 
standard coordinates $(x_1,\ldots, x_n)$ on $\RR^n$.

\item
{\it Lie algebroids:} A Lie algebroid is a vector bundle $A$ over the manifold $X$, together with a Lie bracket $[\cdot,\cdot]$ on its smooth sections and a morphism of vector bundles $\varrho:A\to TX$, 
called the \emph{anchor}, which is bracket-preserving and satisfies the Leibniz rule 
\[
[a,fb] = f[a,b] + (\varrho(a)f)b,
\] 
for all sections $a,b$ of $A$ and $f\in C^\infty(X,\RR)$.

\item
{\it Action Lie algebroid:}
 Given an infinitesimal action of a Lie algebra $\mathfrak{g}$ on
 a manifold $X$, i.e.,, a Lie algebra homomorphism $\rho: \mathfrak{g}
 \rightarrow C^\infty(X,TX)$, the \emph{action Lie algebroid} 
 $\mathfrak{g}\ltimes_\rho X$ is the trivial bundle $\mathfrak{g} \times X\to X$, equipped with the anchor map $\varrho$ defined by 
 $\varrho(a) = (x, \rho(a)(x))$ and
 Lie bracket on sections given by the unique extension of the 
 Lie bracket on constant sections which satisfies the Leibniz rule.
 
\item
{\it Free divisors:} 
A \emph{free divisor} is defined to be a union of smooth closed 
hypersurfaces (i.e., real codimension one submanifolds) with the property
that the sheaf of vector fields tangent to all hypersurfaces 
is locally free.  Given a free divisor $D$ on the $n$-manifold $X$, we let $TX(-\log D)$ denote the rank $n$ vector bundle defined by the above sheaf. It is a Lie algebroid, with Lie bracket inherited from the Lie bracket of vector fields and anchor map defined by the natural inclusion of sheaves.  Its dual bundle is denoted by $T^*X(\log D)$ and the Lie algebroid de Rham complex is usually written as $(\Omega^\bullet(X,\log D), d)$;
it is called the \emph{logarithmic de Rham complex} (See Appendix~\ref{Mazzeo-Melrose}).

\item 
{\it Normal crossing divisors:}
A convenient subclass of free divisors are
  those with only simple normal crossing singularities, meaning that for
  each point $x\in D$, there 
is a chart $(U, \varphi)$ of $X$, $x\in U$, such that $\varphi(D)$ is
a union of a subset of the coordinate hyperplanes in $\RR^n$ intersected
with $\varphi(D)$. The \emph{smooth locus} of $D$ is defined to be the
set of points in $D$ which lie only on a single hypersurface of $D$. The
\emph{singular locus} of $D$ is the complement in $D$ of the smooth 
locus.

\end{enumerate}

\pagebreak
\section{Log affine manifolds}\label{sec2}
\begin{pp}
  A toric Hamiltonian symplectic manifold is equipped with a momentum
  map to the dual of the Lie algebra of the torus: the affine space
  $\Tt^*$.  When generalizing the theory to log symplectic forms, one
  is led to replace the codomain $\Tt^*$ by a \emph{log affine
    manifold}, possibly with boundary and corners.

  \begin{defn}\label{logafin}
  A \textbf{real log affine $n$-manifold} is a manifold $X$ of real dimension
  $n$, equipped with a free divisor $D$ and an isomorphism $\xi$ of Lie
  algebroids between $TX(-\log D)$ and an abelian action algebroid,
  i.e.,
  \begin{equation}
    \label{eq:26}
    \xymatrix{TX(-\log D)\ar[r]^-{\xi} &  \RR^n \ltimes_\rho X},    
  \end{equation}
  where $\rho:\RR^n\to C^\infty(X,TX)$ is a Lie algebra homomorphism.  We
  call a log affine manifold \emph{complete} when the infinitesimal
  action of $\RR^n$ integrates to a group action of $\RR^n$ on $X$.
  The isomorphism~\eqref{eq:26} may be viewed, equivalently, as a
  closed logarithmic 1-form
  \begin{equation}
    \label{eq:57}
    \xi \in \Omega^1(X, \log D)\otimes \RR^n.
  \end{equation}
  \end{defn}

\end{pp}

\begin{remark}
 If $X$ has boundary or corners and if $D$ is a normal crossing divisor, 
 then in order to be complete log affine, the divisor $D$ must include 
 the normal crossing boundary components of $X$. 
\end{remark}

\begin{remark}
  Saito's criterion~\cite{MR586450} provides a convenient way to
  characterize log affine manifolds.  A simple consequence of the
  criterion is the following: if $(X_1,\ldots, X_n)$ are commuting
  vector fields on the $n$-manifold $X$ which are tangent to a normal
  crossing divisor $D$, and if their determinant
  \begin{equation}
    \label{eq:14}
    X_1\wedge\cdots \wedge X_n \in C^\infty(X,\wedge^n TX)
  \end{equation}
  vanishes precisely on $D$ and transversely on its smooth locus, then
  $TX(-\log D)$ is an action Lie algebroid globally generated by the
  $n$ sections $(X_1,\ldots, X_n)$.
\end{remark}

\begin{pp} Assuming that $D$ is a normal crossing divisor, by the Mazzeo-Melrose decomposition~(\ref{Mazzeo-Melrose}), we have the natural isomorphism
\[
H^1(X,\log D)\otimes\RR^n \cong (H^1(X)\oplus \sum_i H^0(D_i))\otimes\RR^n.
\]
As a result, the logarithmic $1$-form~\eqref{eq:57} defines a class in first cohomology with coefficients in $\RR^n$.
\begin{defn}
The \textbf{affine monodromy} of the log affine manifold $(X,D,\xi)$ is defined to be the component of $[\xi]$ in $H^1(X)\otimes\RR^n$.
\end{defn}

The log affine manifolds considered in this paper, typically have trivial affine monodromy, as a consequence of the assumption of an analogue of the Hamiltonian condition in toric symplectic geometry; see Definition~\ref{torlog}.
\end{pp}

\subsection{The dual cubic complex of a tropical welded space} \label{subsec:cubcx}

\begin{pp}
The class of log affine manifolds, useful in the Delzant classification of Hamiltonian toric log symplectic manifolds, is the collection of 
complete log affine manifolds with corners and with trivial affine monodromy. To simplify the terminology, we give it a name.

\begin{defn}
	A \textbf{tropical welded space} is a complete log affine manifold $(X, D, \xi)$ with corners and with trivial monodromy, where 
	$\partial X \subset D$.
\end{defn}

The justification of the terminology comes from the fact that the log affine structure on the closure of each component of $X \setminus D$ is closely related to the the notion of \emph{extended tropicalization} $\mathbf{Trop}(\mathcal{X})$ of a complex toric variety $\mathcal{X}$ introduced by Kajiwara~\cite{MR2428356} and Payne~\cite{MR2511632} (see also~\cite{MR2289207} for similar ideas). We provide more details on this in \S\ref{sec: CaineToric}. As a consequence, the log affine manifold $X$ may be thought as being assembled from the codomains of the extended tropicalization map, hence the name `tropical welded space'.

It turns out that a tropical welded space $(X, D, \xi)$ is determined by a cubic complex, decorated in a certain way with residue vectors along its edges. The key observation is that the normal crossing condition on $D$ gives rise to a cubic complex which is dual to $(X,D)$ in a precise sense.
\end{pp}

\subsubsection{Labeled simplicial fans}

\begin{pp}
	Before we consider general tropical welded spaces, we first deal with the case when the normal crossing divisor coincides with the boundary, 
i.e., $D = \partial X$. We show that in this case, the space is completely characterized by combinatorial data, namely a labeled simplicial fan, defined as follows.

\begin{defn} \label{blocks}
Let $U$ be a real vector space of dimension $n$ and $\sF$ a finite set 
of vectors in $U$.  Let $\sC$ be a collection of linearly independent subsets $A\subset \sF$, each of which is closed under convex hull, in the sense that the convex hull of $A$ does not meet $\sF\setminus A$. The collection $\sC$ must have the property that if $A\in \sC$ then all subsets of $A$ are also contained in $\sC$.  The pair $\Sigma = (\sF, \sC)$ may then be viewed as a \textbf{labeled simplicial fan}: the fan is the set of all strictly convex cones $\IP{A}_+ \subset U$, $A\in\sC$, where $\IP{A}_+$ denotes the convex hull of $A$, and its 1-dimensional cones are labeled by the corresponding elements of $\sF$.
\end{defn}

 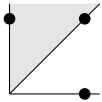
\begin{figure}[H]
    \centering
    \begin{tikzpicture}
      \draw [draw=none,fill=gray, opacity=0.2]
      (0,0) -- (0,1.2) -- (1.2,1.2) -- cycle;
      \draw [draw=none,fill=white, opacity=0.2]
      (0,0) -- (1.2,0) -- (1.2,1.2) -- cycle;
      \draw (0,0) -- (1.2,0);
      \draw (0,0) -- (1.2,1.2);
      \draw (0,0) -- (0,1.2);
      \foreach \Point in {(1,0), (1,1), (0,1)}{
      \draw[fill=black] \Point circle [radius=2pt];
}
    \end{tikzpicture}
    \caption{A labeled simplicial fan with $\sF=\{(1,0),(1,1),(0,1)\}$}\label{fanex}
  \end{figure}
  For example, we may choose a fan labeled by
  $\sF=\{(1,0),(1,1),(0,1)\}$ consisting of one 2-dimensional
  cone and three 1-dimensional cones, as shown in Figure~\ref{fanex}.
	
\begin{prop} \label{prop: semi-local}
	 Let $(X, \partial X, \xi)$ be a $n$-dimensional tropical welded space, where $\partial X$ is the union of a collection $\{D_i\}_{i\in I}$ of smooth boundary components. Let $\xi_i \in \RR^n$ be the residue along the component $D_i \subset \partial X$. Let $\sC$ be the collection of subsets of $\{\xi_i\}_{i\in I}$ such that $\{\xi_{i_1}, \xi_{i_2}, \ldots, \xi_{i_k}\} \in \sC$ if $D_{i_1} \cap D_{i_2} \cap \ldots \cap D_{i_k} \neq \varnothing$. Then $\Sigma = (\{\xi_i\}_{i\in I}, \sC)$ is a labeled simplicial fan.
\end{prop}

\begin{proof}
	We choose a basepoint $x_0 \in X^\circ = X \setminus \partial X$. By completeness and trivial monodromy, we may identify the interior $X^\circ$ with the standard affine space $\RR^n$ via the map
\[
	\RR^n \to X^\circ, \qquad v \mapsto v \cdot x_0,
\]
where $(v, x) \mapsto v \cdot x$ represents the $\RR^n$-action on $X$.

	Now let $x \in D_{i_1} \cap D_{i_2} \cap \ldots \cap D_{i_k}$ and  
	$V$ a neighborhood of $x$, where $\xi$ has the local normal form \eqref{eq:33}. Note that $A = \{\xi_{i_1}, \xi_{i_2}, \ldots, \xi_{i_k}\} \in \sC$ is linearly independent by Proposition~\ref{normlog}.
	
	If $v \in \RR^n$ is a vector in the interior of the cone generated by $\{\xi_{i_1}, \xi_{i_2}, \ldots, \xi_{i_k}\}$, then according to the local normal form in Theorem \ref{umplogaf}, $tv \cdot x_0$ converges to a point in $D_{i_1} \cap D_{i_2} \cap \ldots \cap D_{i_k}$ as $t \to \infty$. Likewise, if $v'$ is a vector in the interior of the cone generated by a different set of vectors $\{\xi_{i'_1}, \xi_{i'_2}, \ldots, \xi_{i'_l}\}$, then $tv' \cdot x_0$ converges to a point in $D_{i'_1} \cap D_{i'_2} \cap \ldots \cap D_{i'_l}$ as $t \to \infty$. This means that the intersection of the interior of the cone generated by $\{\xi_{i_1}, \xi_{i_2}, \ldots, \xi_{i_k}\}$ and the interior of the cone generated by $\{\xi_{i'_1}, \xi_{i'_2}, \ldots, \xi_{i'_l}\}$ is empty, and so $\Sigma = (\{\xi_i\}_{i\in I}, \sC)$ is a labeled simplicial fan.

\end{proof}

\end{pp}

\subsubsection{Partial cubic complexes}

\begin{pp}
The notion of a cubic complex is analogous to that of a simplicial complex, where the building blocks are cubes instead of simplexes. We require a refined notion of a cubic complex, in which a cube need not contain all of its lower-dimensional faces.

For a non-negative integer $k$, a \textbf{$k$-cube} $c$ is the space $[-1, 1]^{k}$ with the convention $[-1, 1]^0 = \{1\}$. A codimension 1 face of $c$ is given by fixing precisely one of the coordinates to be $-1$ or $1$. Other faces of $c$ are obtained by taking intersections of its codimension 1 faces. A face of $c$ is also a cube. For $\ell > k$, let $f$ be a $k$-cube and $c$ an $\ell$-cube. An \textbf{embedding} from $f$ to $c$ is an affine linear map $\varphi_{f,c}$ mapping $f$ isomorphically onto a $k$\--dimensional face of $c$.

\begin{example}
For $\ell > k$, let $f = [-1,1]^k$ and $c = [-1, 1]^\ell$. One possible embedding from $f$ to $c$ is
			$$
				\varphi_{f,c}: [-1,1]^k \to [-1,1]^\ell, \qquad (x_1, \ldots, x_k) \mapsto (x_1, \ldots, x_k, 1, \ldots, 1).
			$$
			All other embeddings from $f$ to $c$ are obtained from this one by permuting the coordinates of the codomain and changing the signs on a subset of these coordinates. For example, the eight maps from $[-1,1]$ to $[-1,1]^2$ send $x$ to $(\pm 1, \pm x)$ and $(\pm x, \pm 1)$, all signs being chosen independently. 
\end{example}

\begin{defn} \label{def: cubcx}
	A \textbf{cubic complex} is a finite collection $\mathsf{C}$ of cubes, together with a finite collection $\mathsf{\Gamma}$ of embeddings between pairs of elements of $\mathsf{C}$, satisfying the following properties:
\begin{enumerate}
	\item for all $c \in \mathsf{C}$, each face of $c$ is the image of a unique $\varphi \in \mathsf{\Gamma}$;
	\item $\mathsf{\Gamma}$ is closed under composition.
\end{enumerate}
A \textbf{partial cubic complex} is a cubic complex $(\mathsf{C}, \mathsf{\Gamma})$, together with the data: for each $k$-cube $c \in \mathsf{C}$, there is a distinguished open subspace $c^\circ \subset c$ where $c^\circ$ is either
$$
	(-1, 1]^{m} \times [-1, 1]^{k-m}, \quad \text{for some} ~ 0 \leq m \leq k,
$$
or the empty set $\varnothing$, for which we require that $\{c^\circ ~|~ c \in \mathsf{C}\}$ is preserved by $\mathsf{\Gamma}$. The \textbf{geometric realization} of the partial cubic complex $(\mathsf{C}, \mathsf{\Gamma})$ is the coproduct
\begin{equation} \label{eq: geomcx}
	|\mathsf{C}| = \coprod_{c \in \mathsf{C}} c^\circ \,/ \sim\,,
\end{equation}
where $x \sim \varphi(x)$ for $\varphi \in \mathsf{\Gamma}$.
\end{defn}

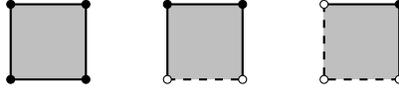
\begin{figure}[H]
    \centering

    \begin{tikzpicture}[scale=1]
\draw[draw=black,fill=lightgray,thick] (-0.5,0.5) -- (-0.5,-0.5) -- (0.5,-0.5) -- (0.5,0.5) -- cycle;
\draw[fill=black] (-0.5,0.5) circle [radius=1.5pt];
\draw[fill=black] (0.5,0.5) circle [radius=1.5pt];
\draw[fill=black] (-0.5,-0.5) circle [radius=1.5pt];
\draw[fill=black] (0.5,-0.5) circle [radius=1.5pt];
    \end{tikzpicture}
    \qquad
    \begin{tikzpicture}[scale=1]
\draw[draw=black,fill=lightgray,thick] (0.5,-0.5) -- (0.5,0.5) -- (-0.5,0.5) -- (-0.5,-0.5);
\draw[black,thick,dashed] (-0.5,-0.5) -- (0.5,-0.5);
\draw[fill=black] (-0.5,0.5) circle [radius=1.5pt];
\draw[fill=black] (0.5,0.5) circle [radius=1.5pt];
\draw[fill=white] (-0.5,-0.5) circle [radius=1.5pt];
\draw[fill=white] (0.5,-0.5) circle [radius=1.5pt];
    \end{tikzpicture}
    \qquad
    \begin{tikzpicture}[scale=1]
\draw[draw=none,fill=lightgray] (0.5,-0.5) -- (0.5,0.5) -- (-0.5,0.5) -- (-0.5,-0.5) -- cycle;
\draw[draw=black,thick] (0.5,-0.5) -- (0.5,0.5) -- (-0.5,0.5);
\draw[draw=black,thick,dashed] (0.5,-0.5) -- (-0.5,-0.5) -- (-0.5,0.5);
\draw[fill=white] (-0.5,0.5) circle [radius=1.5pt];
\draw[fill=black] (0.5,0.5) circle [radius=1.5pt];
\draw[fill=white] (-0.5,-0.5) circle [radius=1.5pt];
\draw[fill=white] (0.5,-0.5) circle [radius=1.5pt];
    \end{tikzpicture}
    \caption{All possibilities of $c^\circ$, except for $\varnothing$, of a 2\--cube $c$}
\end{figure}

For a partial cubic complex $(\mathsf{C}, \mathsf{\Gamma})$, a $0$-cube $v \in \mathsf{C}$, a.k.a. a vertex, is called \textbf{open} if $v^\circ = \varnothing$ and \textbf{closed} if $v^\circ = \{1\}$. When the context is clear, we refer to partial cubic complexes simply as cubic complexes, and we abbreviate $(\mathsf{C}, \mathsf{\Gamma})$ as $\mathsf{C}$. 

Even for the usual cubic complexes, our definition is different from what most authors use in the literature, e.g., Definition 2 in \cite{MR1344739}. For example, two cubes may share more than one face. See also Example \ref{ex: torus}.
\end{pp}

\subsubsection{Admissible decorated cubic complexes}

\begin{pp}\label{pp:dualcubcx}
Let $X$ be a smooth manifold with corners and  $D$ a normal crossing 
divisor that contains $\partial X$. We construct the \textbf{dual cubic complex} $\mathsf{C}_{X, D}$ as follows.

First, we extend $X$ by a collar neighborhood so that the extension $\widetilde{X}$ is a manifold without boundary. We then extend $D$ to a normal crossing divisor $\widetilde{D}$ in $\widetilde{X}$.

Next, using the normal crossing property of $\widetilde{D}$, we construct its dual cubic complex $\mathsf{C}_{\widetilde{X}, \widetilde{D}}$ as follows. Let $\{\widetilde{D}_i\}_{i \in I}$ be the smooth components of $\widetilde{D}$. For $j = 1, \ldots, n$, we define
$$
	\widetilde{D}^j = \bigcup_{i_1, \ldots, i_j \in I} \widetilde{D}_{i_1} \cap \ldots \cap \widetilde{D}_{i_j}.
$$
That is, $\widetilde{D}^j$ is the closure of all codimension $j$ strata $\widetilde{D}_{i_1}, \ldots, \widetilde{D}_{i_j}$ of the stratification of $\widetilde{X}$ defined by $\widetilde D$.  
The vertices of the cubic complex $\mathsf{C}_{\widetilde{X}, \widetilde{D}}$ are defined to be the connected components of  $\widetilde{X} \setminus \widetilde{D}$, and the $j$-dimensional closed cubes are given by the connected components of $\widetilde{D}^{j} \setminus \widetilde{D}^{j+1}$.

Finally, $\mathsf{C}_{X, D}$ is obtained from $\mathsf{C}_{\widetilde{X}, \widetilde{D}}$ by taking the same set of vertices, labeling them as closed or open according to whether they lie in $X\setminus D$ or not, respectively. Note that $\mathsf{C}_{X, D}$ is a partial cubic complex as in Definition \ref{def: cubcx}.

\begin{example}
	Let $X = \{(x, y) \in \RR^2 ~|~ y \leq 0 \}$ be the closed lower half-plane, with normal crossing divisor $D = \{xy = 0\} \subset X$. The corresponding cubic complex $\mathsf{C}_{X, D}$, drawn in Figure~\ref{excc}, is a single partial square $[-1,1] \times [-1,1)$, together with its three edges. The two closed vertices of $\mathsf{C}_{X, D}$ correspond to the two connected components of $X \setminus D$, the three solid edges of $\mathsf{C}_{X, D}$ correspond to the three connected components of $D \setminus \{(0, 0)\}$, and the one square of $\mathsf{C}_{X, D}$ corresponds to the stratum $\{(0, 0)\}$.

\begin{figure}[H]
    \centering
    \begin{tikzpicture}[scale=1]
\draw (-1,0) -- (1,0);
\draw (0,-1) -- (0,0);
      
\draw [draw=none,fill=gray, opacity=0.2]
      (-1,0) -- (1,0) -- (1,-1) -- (-1,-1)-- cycle;
    \end{tikzpicture}
    \qquad
    \begin{tikzpicture}[scale=1]
\draw[draw=none] (-1,0) -- (1,0);
\draw[draw=none] (0,-1) -- (0,0);
\draw[draw=black,fill=lightgray,thick] (-0.5,0.5) -- (-0.5,-0.5) -- (0.5,-0.5) -- (0.5,0.5);
\draw[black,thick,dashed] (-0.5,0.5) -- (0.5,0.5);
\draw[fill=white] (-0.5,0.5) circle [radius=1.5pt];
\draw[fill=white] (0.5,0.5) circle [radius=1.5pt];
\draw[fill=black] (-0.5,-0.5) circle [radius=1.5pt];
\draw[fill=black] (0.5,-0.5) circle [radius=1.5pt];
    \end{tikzpicture}
    \caption{$(X,D)$ and its dual cubic complex.}\label{excc}
\end{figure}
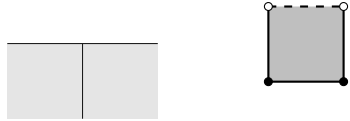

\end{example}

\end{pp}

\begin{pp} \label{pp:admcubcx}

Now let $(X, D, \xi)$ be a tropical welded space. Each smooth component $D_i$ of $D$ has a unique residue vector 
$$
	\xi_i = \mathtt{Res}_i \xi \in U. 
$$
To encode this information in the cubic complex, we note that the 
$1$-dimensional cubes, i.e., the \emph{edges}, are dual to the components 
$D_i$, in such a way that parallel edges are dual to the same $D_i$.  Indeed, each smooth component $D_i$ of $D$ corresponds to a unique equivalence class of parallel edges, where we make the notion of parallel edges into an equivalence relation: the edges $e, e'$ are parallel if 
there is a sequence of edges $(e, e_2,\ldots, e_{n-1}, e')$ such that consecutive edges form parallel edges of a square.  Therefore, we decorate each such equivalence class of edges by the corresponding residue vector $\xi_i$.

Since every connected component $X_j$ of $X \setminus D$ corresponds to a closed vertex of the cubic complex $\mathsf{C}_{X,D}$, and since each edge meeting such a vertex is decorated as above by a residue vector, we see that by Proposition \ref{prop: semi-local}, each closed vertex of $\mathsf{C}_{X,D}$ inherits a labeled simplicial fan. This leads us to the following definition.

\begin{defn}\label{def:admcubcx}
An \textbf{admissible decorated} cubic complex is a partial cubic complex $(\mathsf{C}, \mathsf{\Gamma})$, together with a collection of vectors $\{\xi_i\}_{i\in I}$ in $U=\RR^n$ indexed by the equivalence classes $I$ of parallel edges of $\mathsf{C}$, satisfying the following conditions:
\begin{enumerate}
\item For each cube $c\in\mathsf{C}$ and closed vertex $x\in c$, the set of vectors $A_c = \{\xi_1, \xi_2, \ldots, \xi_k\}$ corresponding to those edges of $c$ meeting $x$ must be linearly independent.

\item For each closed vertex $x$, the collection of sets $\{A_c\}_{x\in c}$ must define a labeled simplicial fan, that is, the residue vectors must define a simplicial fan structure at each closed vertex of the cubic complex.
\end{enumerate}
\end{defn}

\end{pp}

The results of \S\ref{subsec:cubcx} may be summarized as follows:

\begin{prop} \label{prop:admcubcx}
	If $(X, D, \xi)$ is a tropical welded space, then the cubic complex structure $\mathsf{C}_{X, D}$ decorated with the residue vectors $\{\xi_i\}_{i\in I}$ is an admissible decorated cubic complex.
\end{prop}

\subsection{Construction of tropical welded spaces} \label{weld}

\begin{pp}

We now work in the opposite direction. Given an admissible decorated cubic complex $\left( \mathsf{C}, \mathsf{\Gamma}, \{\xi_i\}_{i\in I} \right)$, we construct the corresponding tropical welded space $(X, D, \xi)$. We outline the general ideas before giving details.

For each cube $c \in \mathsf{C}$, we construct a log affine manifold with corners $X_c$. For an embedding $\varphi_{f,c} \in \mathsf{\Gamma}$ of cubes, we construct an embedding $\phi_{f,c}: X_f \hookrightarrow X_c$ of log affine manifolds. That is, we construct a covariant functor from the category of cubes to the category of log affine manifolds with corners. Similar to the geometric realization $|\mathsf{C}|$ as in \eqref{eq: geomcx}, the tropical welded space $(X, D, \xi)$ of the admissible decorated cubic complex $\left( \mathsf{C}, \{\xi_i\}_{i\in I} \right)$ is the coproduct of the log affine manifolds $\{X_c ~|~ c \in \mathsf{C} \}$.

\end{pp}

\subsubsection{The log affine manifold $X_c$ and the embedding $\phi_{f,c}$}

\begin{pp}
Let $\left( \mathsf{C}, \{\xi_i\}_{i\in I} \right)$ be an admissible decorated cubic complex. For a $k$-cube $c \in \mathsf{C}$, let $A_c = \{\xi_1, \xi_2, \ldots, \xi_k\}$ be the set of labeled vectors of its equivalence classes of parallel edges which, by definition, is linearly independent. We first construct a log affine manifold $X_{A_c}$ that contains $X_c$, by the following proposition.

\begin{prop}
	Let $U$ be a real vector space of dimension $n$ and  
$A=\{\xi_1,\ldots, \xi_k\}$ a set of linearly independent vectors in $U$. 
Define the manifold with corners 
\begin{equation} \label{eq: X_A}
	X_A = \left(U\times \RR^k \right)/\,\RR^{k},
\end{equation}
where $\RR^k$ acts via the (free and proper) anti-diagonal action
\begin{equation}
	(s_1,\ldots, s_k)\cdot (x,(\lambda_1,\ldots, \lambda_k)) = (x - \sum_{i=1}^k s_i\xi_i,  (e^{s_1} \lambda_1,\ldots e^{s_k}\lambda_k)).
\end{equation}
The residual $U$-action, given by 
\begin{equation}
	[(x,(\lambda_1,\ldots,\lambda_k))]+u = [(x+u,(\lambda_1,\ldots,\lambda_k))],
\end{equation}
renders $X_A$ into a complete log affine manifold with trivial monodromy. 
\end{prop}

\begin{defn} \label{def: X_c}
Let $\left( \mathsf{C}, \{\xi_i\}_{i\in I} \right)$ be an admissible decorated cubic complex. For a $k$-cube $c \in \mathsf{C}$, where $c^\circ = (-1, 1]^m \times [-1, 1]^{k-m}$ and $A_c = \{\xi_1, \xi_2, \ldots, \xi_k\}$ are the labeled vectors, define $X_c$ by
\begin{equation} \label{eq: X_c}
	X_c = \left(U\times \RR_+^m \times \RR^{k-m} \right)/\,\RR^{k} \subset X_{A_c}\,,
\end{equation}
where $\RR_+ = [0, \infty)$. For $c \in \mathsf{C}$ with $c^\circ = \varnothing$, $X_c = \varnothing$.
\end{defn}

\begin{example} \label{ex: X_c}
	\begin{enumerate}
		\item If $v$ is a closed vertex, then $X_v = U$ with the standard $U$-action on itself.
		\item If $U = \RR^2$ and a square $s$, with $s^\circ = (-1, 1] \times [-1, 1]$, is labeled with $(1, 0)$ and $(1, 1)$, then we have $X_s = \RR^2 \times \left(\RR_+ \times \RR \right)/\,\RR^2$, where the equivalence relation is given by
			$$
			\left( (x_1, x_2), (\lambda_1, \lambda_2) \right) \sim \left( (0, 0), (e^{x_1-x_2}\lambda_1, e^{x_2}\lambda_2) \right).
			$$
			If we fix $x_1 = x_2 = 0$, then $(\lambda_1, \lambda_2)$ becomes a global coordinate chart. The residual $\RR^2$-action is given by
			$$
			\left[(0, 0), (\lambda_1, \lambda_2) \right] + (u_1, u_2) = \left[(0, 0), (e^{u_1-u_2}\lambda_1, e^{u_2}\lambda_2) \right].
			$$
Hence we have
			\begin{equation}
				X_s \cong \RR_+ \times \RR, \qquad \xi_s = \frac{d\lambda_1}{\lambda_1} \otimes (1, 0) + \frac{d\lambda_2}{\lambda_2} \otimes (1, 1).
			\end{equation}
	\end{enumerate}
\end{example}

\end{pp}

\begin{pp}

If $\varphi_{f, c}: f \hookrightarrow c$ is an embedding of cubes, then the embedding $\phi_{f,c}: X_f \hookrightarrow X_c$ is via the standard inclusion of $\RR_+^k$ as an orthant in $\RR^k$. More precisely, we have the following definition.

\begin{defn} \label{def: X_f to X_c}
	Let $f$ and $c$ be cubes of an admissible decorated cubic complex such that $f^\circ = (-1, 1]^{m_f} \times [-1, 1]^{k_f-m_f}$ and $c^\circ = (-1, 1]^{m_c} \times [-1, 1]^{k_c-m_c}$. Let $\varphi_{f, c}: f \hookrightarrow c$ be an embedding of cubes. If we extend
	$$
		\varphi_{f, c}: (-1, 1]^{m_f} \times [-1, 1]^{k_f-m_f} \to (-1, 1]^{m_c} \times [-1, 1]^{k_c-m_c}
	$$
to an affine linear map $\varphi_{f, c}: \RR^{k_f} \to \RR^{k_c}$, then define the induced embedding of log affine manifolds
\begin{equation} \label{eq: X_f to X_c}
	\begin{aligned}
			\phi_{f,c}: & X_f \hookrightarrow X_c, \\
			& \left[x, (\lambda_1,\ldots, \lambda_{k_f}) \right]\mapsto \left[x, \varphi_{f, c} (\lambda_1,\ldots, \lambda_{k_f})\right].
		\end{aligned}
\end{equation}
\end{defn}

\begin{example}
Let $U = \RR^2$ and let $s$ be a square, with $s^\circ = (-1, 1] \times [-1, 1]$, labeled with $(1, 0)$ and $(1, 1)$.
If the closed vertex $v$ is mapped to the point $(1, 1)$ on $s^\circ$, then according to Example \ref{ex: X_c} and \eqref{eq: X_f to X_c}, we have 
$$
	\begin{aligned}
			\phi_{v, s}: & X_v \longrightarrow X_s, \\
			& (x_1, x_2) \mapsto \left[(x_1, x_2), (1, 1) \right] = \left[(0, 0), (e^{x_1-x_2}, e^{x_2}) \right],
		\end{aligned}
$$
or written more compactly,
$$
	\phi_{v, s}: X_v \cong \RR^2 \longrightarrow X_s \cong \RR_+ \times \RR, \qquad  (x_1, x_2) \mapsto \left(e^{x_1-x_2}, e^{x_2}\right).
$$
\end{example}

\end{pp}

\subsubsection{From admissible decorated cubic complexes to tropical welded spaces}

\begin{pp}
Given an admissible decorated cubic complex $\left(\mathsf{C}, \mathsf{\Gamma}, \{\xi_i\}_{i\in I} \right)$, for each cube $c \in \mathsf{C}$, we have a complete log affine manifold $X_c$ as in Definition~\ref{def: X_c}; for each embedding 
$\varphi_{f, c} \in \mathsf{\Gamma}$, we have an embedding $\phi_{f, c}: X_f \to X_c$ as in Definition~\ref{def: X_f to X_c}. Define $X_\mathsf{C}$ as the coproduct
\begin{equation}
	X_\mathsf{C} = \coprod_c X_c \,/ \sim\,,
\end{equation}
where for $x_f \in X_f$ and $x_c \in X_c$, we have $x_f \sim x_c$ if $x_c = \phi_{f, c}(x_f)$.

\begin{prop} \label{prop: grandweld}
	 Given an admissible decorated cubic complex $\left(\mathsf{C}, \mathsf{\Gamma}, \{\xi_i\}_{i\in I} \right)$, the welding procedure detailed above yields a complete log affine manifold with corners and with trivial affine monodromy, 
i.e., $X_\mathsf{C}$ with the natural log affine structure inherited from $X_c$ is a tropical welded space.

Furthermore, if $(X_\mathsf{C}, D, \xi)$ is the resulting tropical welded space, then the dual cubic complex $\mathsf{C}_{X_\mathsf{C}, D}$, decorated with the residue vectors, is $\left(\mathsf{C}, \mathsf{\Gamma}, \{\xi_i\}_{i\in I} \right)$.
\end{prop}

\begin{proof}
That fact $X_\mathsf{C}$ is smooth and log affine is checked locally. Since each $X_c$ for $c \in \mathsf{C}$ is complete, it follows that $(X_\mathsf{C}, D, \xi)$ is also complete.

To show that $(X_\mathsf{C}, D, \xi)$ has trivial affine monodromy, it suffices to check that for each closed loop $\gamma$ on $X$, the Cauchy principal value $PV \int_\gamma \xi$ is zero. Without loss of generality, the base point of $\gamma$ may be chosen to be the origin of $X_v$ for some closed vertex $v \in \mathsf{C}$. It follows that $\gamma$ is homotopic to a curve $\gamma'$ with the following properties:
\begin{enumerate}
	\item $\gamma' = \gamma_1 \gamma_2 \ldots \gamma_k$;
	\item each segment $\gamma_j$ connects the origins of $X_{v_{j-1}} \cong U$ and $X_{v_{j}} \cong U$, where the two vertices $v_{j-1}$ and $v_{j}$ are connected by an edge, and $v_0 = v_n = v$.
\end{enumerate}
From the construction of $(X_\mathsf{C}, D, \xi)$, we have
$$
	PV \int_\gamma \xi = \sum_j PV \int_{\gamma_j} \xi = 0.
$$
It is also clear from the construction that $\mathsf{C}_{X_\mathsf{C}, D}$ decorated with the residue vectors is indeed $\left( \mathsf{C}, \{\xi_i\}_{i\in I} \right)$.
\end{proof}

\end{pp}

\begin{pp}

Combining the results of Proposition~\ref{prop:admcubcx} and Proposition~\ref{prop: grandweld}, we have the following result.

\begin{theorem} \label{thm: weld}
	If $(X, D, \xi)$ is a tropical welded space, then the cubic complex structure $\mathsf{C}_{X, D}$ decorated with the residue vectors $\{\xi_i\}_{i\in I}$ is an admissible decorated cubic complex.

Furthermore, this yields a bijection between the isomorphism classes of tropical welded spaces and the isomorphism classes of admissible decorated cubic complexes.
\end{theorem}

\end{pp}

\subsection{Examples}

We now present three illustrative examples of welding to produce
nontrivial log affine structures on $S^2$, $T^2$, and an
orientable surface of genus two. In these examples we omit a detailed
discussion of the resulting logarithmic one-form~\eqref{eq:57} and its
residues or periods.

\begin{example}\label{spherethreecircles}
Consider the $\RR^3$-action on $\RR^3\setminus\{0\}$ by rescaling:
\begin{equation}
	(t, u, v) \cdot (x,y,z) = (e^t x, e^u y, e^v z).
\end{equation}
If we express $S^2$ as the quotient of $\RR^3\setminus\{0\}$ by the diagonal $\RR$-action
\begin{equation}
  \label{eq:22}
  t\cdot(x,y,z) = e^t(x,y,z),
\end{equation}
then $S^2$ carries a residual $\RR^2$-action, and therefore a
log affine structure with degeneracy locus given by the intersection
with the coordinate hyperplanes, a normal crossing divisor which
divides the sphere into 8 contractible regions. 
The dual cubic complex is the standard cube including all its faces.

\begin{figure}[H]
\centering
\tdplotsetmaincoords{60}{30}
\begin{tikzpicture}[scale=.4]
\begin{scope}[tdplot_main_coords]
\draw[black,dashed] (0,0,0) pic{carc=-50:220:.8};
\draw[black,thick] (0,0,0) pic{carc=-150:30:.8};
\end{scope}
\tdplotsetrotatedcoords{0}{90}{0}
\begin{scope}[tdplot_rotated_coords]
\draw[black,dashed] (0,0,0) pic{carc=-30:140:.8};
\draw[black,thick] (0,0,0) pic{carc=140:330:.8};
\end{scope}
\tdplotsetrotatedcoords{90}{90}{-90}
\begin{scope}[tdplot_rotated_coords]
\draw[black,dashed] (0,0,0) pic{carc=37:220:.8};
\draw[black,thick] (0,0,0) pic{carc=-140:37:.8};
\end{scope}
\begin{scope}[tdplot_screen_coords]
\draw[black,thick] (0,0) circle (2);
\end{scope}

\hspace{4cm}
\newlength{\hshift}
\setlength{\hshift}{-2cm}
\newlength{\vshift}
\setlength{\vshift}{-2cm}

    \coordinate (A1) at (0cm+\hshift, 0cm+\vshift);
    \coordinate (A2) at (0cm+\hshift, 3cm+\vshift);
    \coordinate (A3) at (3cm+\hshift, 3cm+\vshift);
    \coordinate (A4) at (3cm+\hshift, 0+\vshift);
    \coordinate (B1) at (0.9cm+\hshift, 0.9cm+\vshift);
    \coordinate (B2) at (0.9cm+\hshift, 3.9cm+\vshift);
    \coordinate (B3) at (3.9cm+\hshift, 3.9cm+\vshift);
    \coordinate (B4) at (3.9cm+\hshift, 0.9cm+\vshift);

    \draw (A1) -- (A2);
    \draw (A2) -- (A3);
    \draw (A3) -- (A4);
    \draw (A4) -- (A1);

    \draw[dashed] (A1) -- (B1);
    \draw (A2) -- (B2);
    \draw (A3) -- (B3);
    \draw (A4) -- (B4);

    \draw[dashed] (B1) -- (B2);
    \draw (B2) -- (B3);
    \draw (B4) -- (B3);
    \draw[dashed] (B1) -- (B4);

	\node at (-0.4,-2.2)  {\tiny $\alpha$};
	\node at (-0.4,-2.2+3)  {\tiny $\alpha$};
	\node at (-0.4+0.8,-2.2+0.95)  {\tiny $\alpha$};
	\node at (-0.4+0.8,-2.2+0.95+3)  {\tiny $\alpha$};

	\node at (-2.2,-0.5)  {\tiny $\beta$};
	\node at (-2.2+3,-0.5)  {\tiny $\beta$};
	\node at (-2.2+0.9,-0.5+0.8)  {\tiny $\beta$};
	\node at (-2.2+0.9+3,-0.5+0.8)  {\tiny $\beta$};

	\node at (-1.6,-1.4)  {\tiny $\gamma$};
	\node at (-1.6+3,-1.4)  {\tiny $\gamma$};
	\node at (-1.6,-1.4+3)  {\tiny $\gamma$};
	\node at (-1.6+3,-1.4+3)  {\tiny $\gamma$};

\draw[fill=black] (0cm+\hshift, 0cm+\vshift) circle [radius=1.5pt];
\draw[fill=black] (0cm+\hshift, 3cm+\vshift) circle [radius=1.5pt];
\draw[fill=black] (3cm+\hshift, 3cm+\vshift) circle [radius=1.5pt];
\draw[fill=black] (3cm+\hshift, 0+\vshift) circle [radius=1.5pt];
\draw[fill=black] (0.9cm+\hshift, 0.9cm+\vshift) circle [radius=1.5pt];
\draw[fill=black] (0.9cm+\hshift, 3.9cm+\vshift) circle [radius=1.5pt];
\draw[fill=black] (3.9cm+\hshift, 3.9cm+\vshift) circle [radius=1.5pt];
\draw[fill=black] (3.9cm+\hshift, 0.9cm+\vshift) circle [radius=1.5pt];

\draw [draw=none,fill=gray, opacity=0.2] (A1)--(A2)--(A3)--(A4)--(A1);
\draw [draw=none,fill=gray, opacity=0.2] (A2)--(A3)--(B3)--(B2)--(A2);
\draw [draw=none,fill=gray, opacity=0.2] (A3)--(A4)--(B4)--(B3)--(A3);

\end{tikzpicture}
\hspace{10em}
\caption{Log affine structure on $S^2$ with degeneracy along three circles and its dual cubic complex}
\end{figure}
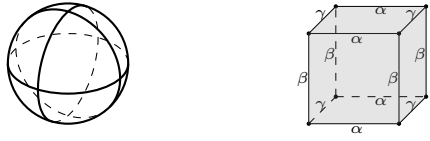

Each of the 8 vertices of the standard cube inherits a labeled simplicial fan. These are illustrated in Figure~\ref{sphaf}, where the Greek letters represent the labeled vectors along the edges. We add the subscripts to the Greek letters in order to distinguish the different parallel edges.

\begin{figure}[H]
    \centering
    \begin{tikzpicture}[scale=.8]
      \draw [-] (0,0) -- (-.8,-.8);
      \node at (-.3,-.6)  {$\gamma_1$};
      \node at (-.2,.6)  {$\beta_1$};
      \node at (.8,.2) {$\alpha_1$};
            \draw [-] (0,0) -- (0,1);
      \draw [-] (0,0) -- (1,0);
          \draw [draw=none,fill=gray, opacity=0.2]
      (0,0) -- (0,1) -- (1,1) -- (1,0)-- cycle;
          \draw [draw=none,fill=gray, opacity=0.2]
      (0,0) -- (0,1) -- (-.8,.2) -- (-.8,-.8)-- cycle;
          \draw [draw=none,fill=gray, opacity=0.2]
      (0,0) -- (1,0) -- (.2,-.8) -- (-.8,-.8)-- cycle;
    \draw[fill=black] (-.6,-.6) circle [radius=1pt];
	\draw[fill=black] (.8,0) circle [radius=1pt];	
	\draw[fill=black] (0,.8) circle [radius=1pt];	
    \end{tikzpicture}
\hspace{1em}
    \begin{tikzpicture}[scale=.8]
      \draw [-] (0,0) -- (-.8,-.8);
      \node at (-.3,-.6)  {$\gamma_2$};
      \node at (-.2,.6)  {$\beta_2$};
      \node at (.8,.2) {$\alpha_1$};
      \draw [-] (0,0) -- (0,1);
      \draw [-] (0,0) -- (1,0);
          \draw [draw=none,fill=gray, opacity=0.2]
      (0,0) -- (0,1) -- (1,1) -- (1,0)-- cycle;
          \draw [draw=none,fill=gray, opacity=0.2]
      (0,0) -- (0,1) -- (-.8,.2) -- (-.8,-.8)-- cycle;
          \draw [draw=none,fill=gray, opacity=0.2]
      (0,0) -- (1,0) -- (.2,-.8) -- (-.8,-.8)-- cycle;
    \draw[fill=black] (-.6,-.6) circle [radius=1pt];
	\draw[fill=black] (.8,0) circle [radius=1pt];	
	\draw[fill=black] (0,.8) circle [radius=1pt];
	\end{tikzpicture}
\hspace{1em}
    \begin{tikzpicture}[scale=.8]
      \draw [-] (0,0) -- (-.8,-.8);
      \node at (-.3,-.6)  {$\gamma_4$};
      \node at (-.2,.6)  {$\beta_1$};
      \node at (.8,.2) {$\alpha_2$};
      \draw [-] (0,0) -- (0,1);
      \draw [-] (0,0) -- (1,0);
          \draw [draw=none,fill=gray, opacity=0.2]
      (0,0) -- (0,1) -- (1,1) -- (1,0)-- cycle;
          \draw [draw=none,fill=gray, opacity=0.2]
      (0,0) -- (0,1) -- (-.8,.2) -- (-.8,-.8)-- cycle;
          \draw [draw=none,fill=gray, opacity=0.2]
      (0,0) -- (1,0) -- (.2,-.8) -- (-.8,-.8)-- cycle;
    \draw[fill=black] (-.6,-.6) circle [radius=1pt];
	\draw[fill=black] (.8,0) circle [radius=1pt];	
	\draw[fill=black] (0,.8) circle [radius=1pt];
	\end{tikzpicture}
\hspace{1em}
    \begin{tikzpicture}[scale=.8]
      \draw [-] (0,0) -- (-.8,-.8);
      \node at (-.3,-.6)  {$\gamma_1$};
      \node at (-.2,.6)  {$\beta_4$};
      \node at (.8,.2) {$\alpha_4$};
      \draw [-] (0,0) -- (0,1);
      \draw [-] (0,0) -- (1,0);
          \draw [draw=none,fill=gray, opacity=0.2]
      (0,0) -- (0,1) -- (1,1) -- (1,0)-- cycle;
          \draw [draw=none,fill=gray, opacity=0.2]
      (0,0) -- (0,1) -- (-.8,.2) -- (-.8,-.8)-- cycle;
          \draw [draw=none,fill=gray, opacity=0.2]
      (0,0) -- (1,0) -- (.2,-.8) -- (-.8,-.8)-- cycle;
    \draw[fill=black] (-.6,-.6) circle [radius=1pt];
	\draw[fill=black] (.8,0) circle [radius=1pt];	
	\draw[fill=black] (0,.8) circle [radius=1pt];
	\end{tikzpicture}
\\
    \begin{tikzpicture}[scale=.8]
      \draw [-] (0,0) -- (-.8,-.8);
      \node at (-.3,-.6)  {$\gamma_3$};
      \node at (-.2,.6)  {$\beta_2$};
      \node at (.8,.2) {$\alpha_2$};
      \draw [-] (0,0) -- (0,1);
      \draw [-] (0,0) -- (1,0);
          \draw [draw=none,fill=gray, opacity=0.2]
      (0,0) -- (0,1) -- (1,1) -- (1,0)-- cycle;
          \draw [draw=none,fill=gray, opacity=0.2]
      (0,0) -- (0,1) -- (-.8,.2) -- (-.8,-.8)-- cycle;
          \draw [draw=none,fill=gray, opacity=0.2]
      (0,0) -- (1,0) -- (.2,-.8) -- (-.8,-.8)-- cycle;
    \draw[fill=black] (-.6,-.6) circle [radius=1pt];
	\draw[fill=black] (.8,0) circle [radius=1pt];	
	\draw[fill=black] (0,.8) circle [radius=1pt];
	\end{tikzpicture}
\hspace{1em}
    \begin{tikzpicture}[scale=.8]
      \draw [-] (0,0) -- (-.8,-.8);
      \node at (-.3,-.6)  {$\gamma_2$};
      \node at (-.2,.6)  {$\beta_3$};
      \node at (.8,.2) {$\alpha_4$};
      \draw [-] (0,0) -- (0,1);
      \draw [-] (0,0) -- (1,0);
          \draw [draw=none,fill=gray, opacity=0.2]
      (0,0) -- (0,1) -- (1,1) -- (1,0)-- cycle;
          \draw [draw=none,fill=gray, opacity=0.2]
      (0,0) -- (0,1) -- (-.8,.2) -- (-.8,-.8)-- cycle;
          \draw [draw=none,fill=gray, opacity=0.2]
      (0,0) -- (1,0) -- (.2,-.8) -- (-.8,-.8)-- cycle;
    \draw[fill=black] (-.6,-.6) circle [radius=1pt];
	\draw[fill=black] (.8,0) circle [radius=1pt];	
	\draw[fill=black] (0,.8) circle [radius=1pt];
	\end{tikzpicture}
\hspace{1em}
    \begin{tikzpicture}[scale=.8]
      \draw [-] (0,0) -- (-.8,-.8);
      \node at (-.3,-.6)  {$\gamma_4$};
      \node at (-.2,.6)  {$\beta_4$};
      \node at (.8,.2) {$\alpha_3$};
      \draw [-] (0,0) -- (0,1);
      \draw [-] (0,0) -- (1,0);
          \draw [draw=none,fill=gray, opacity=0.2]
      (0,0) -- (0,1) -- (1,1) -- (1,0)-- cycle;
          \draw [draw=none,fill=gray, opacity=0.2]
      (0,0) -- (0,1) -- (-.8,.2) -- (-.8,-.8)-- cycle;
          \draw [draw=none,fill=gray, opacity=0.2]
      (0,0) -- (1,0) -- (.2,-.8) -- (-.8,-.8)-- cycle;
    \draw[fill=black] (-.6,-.6) circle [radius=1pt];
	\draw[fill=black] (.8,0) circle [radius=1pt];	
	\draw[fill=black] (0,.8) circle [radius=1pt];
	\end{tikzpicture}
\hspace{1em}
    \begin{tikzpicture}[scale=.8]
      \draw [-] (0,0) -- (-.8,-.8);
      \node at (-.3,-.6)  {$\gamma_3$};
      \node at (-.2,.6)  {$\beta_3$};
      \node at (.8,.2) {$\alpha_3$};
      \draw [-] (0,0) -- (0,1);
      \draw [-] (0,0) -- (1,0);
          \draw [draw=none,fill=gray, opacity=0.2]
      (0,0) -- (0,1) -- (1,1) -- (1,0)-- cycle;
          \draw [draw=none,fill=gray, opacity=0.2]
      (0,0) -- (0,1) -- (-.8,.2) -- (-.8,-.8)-- cycle;
          \draw [draw=none,fill=gray, opacity=0.2]
      (0,0) -- (1,0) -- (.2,-.8) -- (-.8,-.8)-- cycle;
    \draw[fill=black] (-.6,-.6) circle [radius=1pt];
	\draw[fill=black] (.8,0) circle [radius=1pt];	
	\draw[fill=black] (0,.8) circle [radius=1pt];
	\end{tikzpicture}
\caption{Labelled simplicial fans of the tropical welded space $S^2$}\label{sphaf}
\end{figure}
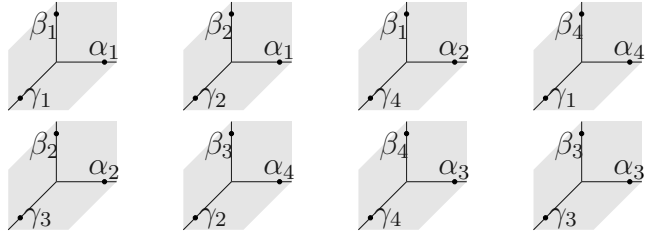
\end{example}

\begin{example} \label{ex: torus}

	Let $\theta \in \RR / 2\pi \ZZ$ be the coordinate on $S^1$. The vector field
$$
	X = \sin \theta \frac{\partial}{\partial \theta}
$$
has two non-degenerate zeros. We equip the 2-torus $S^1\times S^1$ with a log affine structure given by the action of $\RR^2$ generated by the vector fields $V_1=(X,0)$ and $V_2=(0,X)$. In these coordinates, the $\RR^2$-valued logarithmic 1-form $\xi$ on $S^1 \times S^1$ takes the form:
$$
	\xi = \left(\frac{d \theta_1}{\sin \theta_1}, \frac{d \theta_2}{\sin \theta_2}\right).
$$
The determinant $V_1\wedge V_2$ then defines a
  normal crossing divisor consisting of four circles, dividing the
  torus into four square regions.

\begin{figure}[H]\centering
	\begin{tikzpicture}[scale=.8]
	
	\draw (-3,-1) -- (-3,1);
	\node at (-3.2, 0)  {$a$};

	\draw (-1,-1) -- (-1,1);
	\node at (-0.55, 0)  {$a^{-1}$};

	\draw (-3,1) -- (-1,1);
	\node at (-2, 1.3)  {$b^{-1}$};

	\draw (-3,-1) -- (-1,-1);
	\node at (-2, -1.3)  {$b$};

	\draw (-3,0) -- (-1,0);
	\draw (-2,-1) -- (-2,1);
	
	\draw [lightgr] (1,-1) -- (1,1);
	\node at (0.8, 0)  {$a$};

	\draw [lightgr] (3,-1) -- (3,1);
	\node at (3.45, 0)  {$a^{-1}$};

	\draw [lightgr] (1,1) -- (3,1);
	\node at (2, 1.3)  {$b^{-1}$};

	\draw [lightgr] (1,-1) -- (3,-1);
	\node at (2, -1.3)  {$b$};

	\draw (1,0.5) -- (3,0.5);
	\draw (1,-0.5) -- (3,-0.5);
	\draw (1.5,-1) -- (1.5,1);
	\draw (2.5,-1) -- (2.5,1);
	\draw [draw=none,fill=gray, opacity=0.2]
	(3,1) -- (3,-1) -- (1,-1) -- (1,1)-- cycle;

	\draw[fill=black] (1.5, -0.5) circle [radius=1.5pt];
	\draw[fill=black] (1.5, 0.5) circle [radius=1.5pt];
	\draw[fill=black] (2.5, -0.5) circle [radius=1.5pt];
	\draw[fill=black] (2.5, 0.5) circle [radius=1.5pt];

	\node at (2,0.4)  {\tiny $\alpha$};
	\node at (1.25,0.4)  {\tiny $\alpha$};
	\node at (2.75,0.4)  {\tiny $\alpha$};

	\node at (2,-0.6)  {\tiny $\alpha$};
	\node at (1.25,-0.6)  {\tiny $\alpha$};
	\node at (2.75,-0.6)  {\tiny $\alpha$};

	\node at (1.4,0.75)  {\tiny $\beta$};
	\node at (1.4,0)  {\tiny $\beta$};
	\node at (1.4,-0.75)  {\tiny $\beta$};

	\node at (2.4,0.75)  {\tiny $\beta$};
	\node at (2.4,0)  {\tiny $\beta$};
	\node at (2.4,-0.75)  {\tiny $\beta$};

	\end{tikzpicture}
\caption{$T^2$ with its degeneracy divisor and the dual cubic complex}
\end{figure}
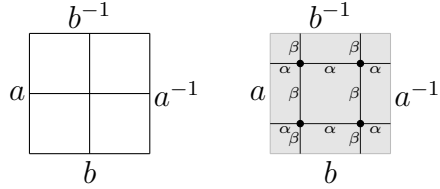

Each of the 4 vertices of the dual cubic complex inherits a labeled simplicial fan. These are illustrated in Figure~\ref{toraf}, where the Greek letters represent the labeled vectors along the edges. We add the subscripts to the Greek letters in order to distinguish the different parallel edges.
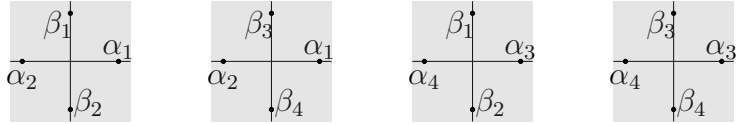
\begin{figure}[H]
    \centering
    \begin{tikzpicture}[scale=.8]
      \node at (-.2,.6)  {$\beta_1$};
      \node at (.8,.2) {$\alpha_1$};
      \node at (-.8,-.3) {$\alpha_2$};
      \node at (.3,-.7) {$\beta_2$};
      \draw [-] (0,0) -- (0,1);
      \draw [-] (0,0) -- (1,0);
      \draw [-] (0,0) -- (0,-1);
      \draw [-] (0,0) -- (-1,0);
          \draw [draw=none,fill=gray, opacity=0.2]
      (-1,-1) -- (-1,1) -- (1,1) -- (1,-1)-- cycle;
    \draw[fill=black] (-.8,0) circle [radius=1pt];
	\draw[fill=black] (.8,0) circle [radius=1pt];	
	\draw[fill=black] (0,.8) circle [radius=1pt];
	\draw[fill=black] (0,-.8) circle [radius=1pt];
	
    \end{tikzpicture}
\hspace{1em}
    \begin{tikzpicture}[scale=.8]
      \node at (-.2,.6)  {$\beta_3$};
      \node at (.8,.2) {$\alpha_1$};
      \node at (-.8,-.3) {$\alpha_2$};
      \node at (.3,-.7) {$\beta_4$};
      \draw [-] (0,0) -- (0,1);
      \draw [-] (0,0) -- (1,0);
      \draw [-] (0,0) -- (0,-1);
      \draw [-] (0,0) -- (-1,0);
          \draw [draw=none,fill=gray, opacity=0.2]
      (-1,-1) -- (-1,1) -- (1,1) -- (1,-1)-- cycle;
\draw[fill=black] (-.8,0) circle [radius=1pt];
	\draw[fill=black] (.8,0) circle [radius=1pt];	
	\draw[fill=black] (0,.8) circle [radius=1pt];
	\draw[fill=black] (0,-.8) circle [radius=1pt];
	
    \end{tikzpicture}
\hspace{1em}
    \begin{tikzpicture}[scale=.8]
      \node at (-.2,.6)  {$\beta_1$};
      \node at (.8,.2) {$\alpha_3$};
      \node at (-.8,-.3) {$\alpha_4$};
      \node at (.3,-.7) {$\beta_2$};
      \draw [-] (0,0) -- (0,1);
      \draw [-] (0,0) -- (1,0);
      \draw [-] (0,0) -- (0,-1);
      \draw [-] (0,0) -- (-1,0);
          \draw [draw=none,fill=gray, opacity=0.2]
      (-1,-1) -- (-1,1) -- (1,1) -- (1,-1)-- cycle;
	
	\draw[fill=black] (-.8,0) circle [radius=1pt];
	\draw[fill=black] (.8,0) circle [radius=1pt];	
	\draw[fill=black] (0,.8) circle [radius=1pt];
	\draw[fill=black] (0,-.8) circle [radius=1pt];
	
    \end{tikzpicture}
\hspace{1em}
    \begin{tikzpicture}[scale=.8]
      \node at (-.2,.6)  {$\beta_3$};
      \node at (.8,.2) {$\alpha_3$};
      \node at (-.8,-.3) {$\alpha_4$};
      \node at (.3,-.7) {$\beta_4$};
      \draw [-] (0,0) -- (0,1);
      \draw [-] (0,0) -- (1,0);
      \draw [-] (0,0) -- (0,-1);
      \draw [-] (0,0) -- (-1,0);
	\draw [draw=none,fill=gray, opacity=0.2]
	(-1,-1) -- (-1,1) -- (1,1) -- (1,-1)-- cycle;

	\draw[fill=black] (-.8,0) circle [radius=1pt];
	\draw[fill=black] (.8,0) circle [radius=1pt];	
	\draw[fill=black] (0,.8) circle [radius=1pt];
	\draw[fill=black] (0,-.8) circle [radius=1pt];
	
    \end{tikzpicture}

\caption{Labelled simplicial fans of the tropical welded space $T^2$}\label{toraf}
\end{figure}
\end{example}

\begin{example}\label{hex}

We consider the orientable genus 2 surface with six degeneracy circles as illustrated on the left part of Figure~\ref{fig: genus2surf}. (See also Figure~\ref{fig:genus2}.)
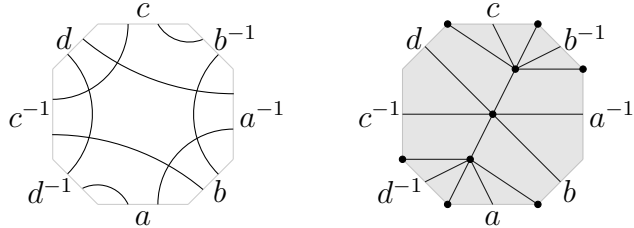
\begin{figure}[H]
    \centering

    \begin{tikzpicture}[scale=.6]
\draw [lightgr] (2,-1) -- (2,1);
	\node at (2.65, 0)  {$a^{-1}$};

	\draw [lightgr] (2,1) -- (1,2);
	\node at (2.05, 1.7)  {$b^{-1}$};

	\draw [lightgr] (1,2) -- (-1,2);
	\node at (0, 2.3)  {$c$};

	\draw [lightgr] (-1,2) -- (-2,1);
	\node at (-1.75, 1.7)  {$d$};

	\draw [lightgr] (-2,1) -- (-2,-1);
	\node at (-2.5, 0)  {$c^{-1}$};

	\draw [lightgr] (-2,-1) -- (-1,-2);
	\node at (-2.05, -1.7)  {$d^{-1}$};

	\draw [lightgr] (-1,-2) -- (1,-2);
	\node at (0, -2.3)  {$a$};

	\draw [lightgr] (1,-2) -- (2,-1);
	\node at (1.7, -1.7)  {$b$};

	\draw (0.33,2) arc (-150:-68:0.8);
	\draw (-0.33,-2) arc (25:116:0.75);

	\draw (1.33,-1.67) arc (50:90:5.2);
	\draw (-1.33,1.67) arc (230:270:5.2);

	\draw (0.33,-2) arc (180:90:1.67);
	\draw (-0.33,2) arc (0:-90:1.67);

	\draw (1.67,-1.33) arc (225:135:1.875);
	\draw (-1.67,1.33) arc (45:-45:1.875);
    \end{tikzpicture}
\hspace{1em}
    \begin{tikzpicture}[scale=.6]

	\draw [lightgr] (2,-1) -- (2,1);
	\node at (2.65, 0)  {$a^{-1}$};

	\draw [lightgr] (2,1) -- (1,2);
	\node at (2.05, 1.7)  {$b^{-1}$};

	\draw [lightgr] (1,2) -- (-1,2);
	\node at (0, 2.3)  {$c$};

	\draw [lightgr] (-1,2) -- (-2,1);
	\node at (-1.75, 1.7)  {$d$};

	\draw [lightgr] (-2,1) -- (-2,-1);
	\node at (-2.5, 0)  {$c^{-1}$};

	\draw [lightgr] (-2,-1) -- (-1,-2);
	\node at (-2.05, -1.7)  {$d^{-1}$};

	\draw [lightgr] (-1,-2) -- (1,-2);
	\node at (0, -2.3)  {$a$};

	\draw [lightgr] (1,-2) -- (2,-1);
	\node at (1.7, -1.7)  {$b$};
	
	\draw [-] (1,2) -- (-1,-2);
	\draw [-] (2,0) -- (-2,0);
	\draw [-] (-1.5,1.5) -- (1.5,-1.5);

	\draw [-] (0.5,1) -- (2,1);
	\draw [-] (0.5,1) -- (-1,2);
	\draw [-] (0.5,1) -- (1.5,1.5);
	\draw [-] (0.5,1) -- (0,2);

	\draw [-] (-0.5,-1) -- (-2,-1);
	\draw [-] (-0.5,-1) -- (1,-2);
	\draw [-] (-0.5,-1) -- (-1.5,-1.5);
	\draw [-] (-0.5,-1) -- (0,-2);




	\draw [draw=none,fill=gray, opacity=0.2]
	(2,-1) -- (2,1) -- (1,2) -- (-1,2) -- (-2,1) -- (-2,-1) -- (-1,-2) -- (1,-2) -- cycle;

	\draw[fill=black] (1,2) circle [radius=2pt];
	\draw[fill=black] (0.5,1) circle [radius=2pt];	
	\draw[fill=black] (0,0) circle [radius=2pt];
	\draw[fill=black] (-1,-2) circle [radius=2pt];
	\draw[fill=black] (-0.5,-1) circle [radius=2pt];

	\draw[fill=black] (2,1) circle [radius=2pt];
	\draw[fill=black] (-1,2) circle [radius=2pt];
	\draw[fill=black] (-2,-1) circle [radius=2pt];
	\draw[fill=black] (1,-2) circle [radius=2pt];
    \end{tikzpicture}
\caption{Genus 2 surface with six degeneracy circles and the dual cubic complex}
\label{fig: genus2surf}
\end{figure}
Its dual cubic complex, as illustrated on the right part of 
Figure~\ref{fig: genus2surf}, may be made into an admissible decorated 
cubic complex. Note that the six dots on the boundary of the octagon are the same point of the
genus 2 surface, so there are 4 vertices on the dual cubic complex.

Figure~\ref{hexfig} illustrates the labeled simplicial fans of the 4 vertices, where we add the subscripts to distinguish the different parallel edges. Topologically, the six degeneracy circles on the orientable genus 2 surface gives a tiling of the surface by four hexagons, and Figure~\ref{hexfig} indicates how the four hexagons are glued along edges.

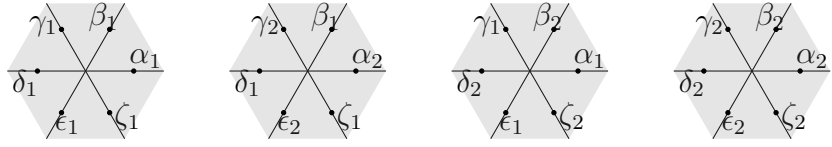
\begin{figure}[H]
    \centering
    \begin{tikzpicture}[scale=.8]
      \node at (1,.2)  {$\alpha_1$};
      \node at (.3,.9) {$\beta_1$};
      \node at (-.7,.75) {$\gamma_1$};
      \node at (-1,-.2) {$\delta_1$};
      \node at (-.3,-.9) {$\epsilon_1$};
      \node at (.7,-.75) {$\zeta_1$};
      
      \draw [-,scale = 1.3] (0,0) -- (1,0);
      \draw [-,scale = 1.3] (0,0) -- (.5,0.866);
      \draw [-,scale = 1.3] (0,0) -- (-.5,0.866);
      \draw [-,scale = 1.3] (0,0) -- (-1,0);
      \draw [-,scale = 1.3] (0,0) -- (-.5,-0.866);
      \draw [-,scale = 1.3] (0,0) -- (.5,-0.866);
      \draw [draw=none,fill=gray, opacity=0.2,scale = 1.3]
      (1,0) -- (.5,.866) -- (-.5,.866) -- (-1,0) -- (-.5,-.866) -- (.5, -.866) -- cycle;
      \draw[fill=black] (.8,0) circle [radius=1pt];
      \draw[fill=black] (.8*.5,.8*0.866) circle [radius=1pt];
      \draw[fill=black] (.8*-.5,.8*0.866) circle [radius=1pt];
      \draw[fill=black] (.8*-1,0) circle [radius=1pt];
      \draw[fill=black] (.8*-.5,.8*-0.866) circle [radius=1pt];
      \draw[fill=black] (.8*.5,.8*-0.866) circle [radius=1pt];
      
    \end{tikzpicture}
\hspace{1em}
    \begin{tikzpicture}[scale=.8]
      \node at (1,.2)  {$\alpha_2$};
      \node at (.3,.9) {$\beta_1$};
      \node at (-.7,.75) {$\gamma_2$};
      \node at (-1,-.2) {$\delta_1$};
      \node at (-.3,-.9) {$\epsilon_2$};
      \node at (.7,-.75) {$\zeta_1$};
      
      \draw [-,scale = 1.3] (0,0) -- (1,0);
      \draw [-,scale = 1.3] (0,0) -- (.5,0.866);
      \draw [-,scale = 1.3] (0,0) -- (-.5,0.866);
      \draw [-,scale = 1.3] (0,0) -- (-1,0);
      \draw [-,scale = 1.3] (0,0) -- (-.5,-0.866);
      \draw [-,scale = 1.3] (0,0) -- (.5,-0.866);
      \draw [draw=none,fill=gray, opacity=0.2,scale = 1.3]
      (1,0) -- (.5,.866) -- (-.5,.866) -- (-1,0) -- (-.5,-.866) -- (.5, -.866) -- cycle;
      \draw[fill=black] (.8,0) circle [radius=1pt];
      \draw[fill=black] (.8*.5,.8*0.866) circle [radius=1pt];
      \draw[fill=black] (.8*-.5,.8*0.866) circle [radius=1pt];
      \draw[fill=black] (.8*-1,0) circle [radius=1pt];
      \draw[fill=black] (.8*-.5,.8*-0.866) circle [radius=1pt];
      \draw[fill=black] (.8*.5,.8*-0.866) circle [radius=1pt];
    \end{tikzpicture}
\hspace{1em}
    \begin{tikzpicture}[scale=.8]
      \node at (1,.2)  {$\alpha_1$};
      \node at (.3,.9) {$\beta_2$};
      \node at (-.7,.75) {$\gamma_1$};
      \node at (-1,-.2) {$\delta_2$};
      \node at (-.3,-.9) {$\epsilon_1$};
      \node at (.7,-.75) {$\zeta_2$};
      
      \draw [-,scale = 1.3] (0,0) -- (1,0);
      \draw [-,scale = 1.3] (0,0) -- (.5,0.866);
      \draw [-,scale = 1.3] (0,0) -- (-.5,0.866);
      \draw [-,scale = 1.3] (0,0) -- (-1,0);
      \draw [-,scale = 1.3] (0,0) -- (-.5,-0.866);
      \draw [-,scale = 1.3] (0,0) -- (.5,-0.866);
      \draw [draw=none,fill=gray, opacity=0.2,scale = 1.3]
      (1,0) -- (.5,.866) -- (-.5,.866) -- (-1,0) -- (-.5,-.866) -- (.5, -.866) -- cycle;
      \draw[fill=black] (.8,0) circle [radius=1pt];
      \draw[fill=black] (.8*.5,.8*0.866) circle [radius=1pt];
      \draw[fill=black] (.8*-.5,.8*0.866) circle [radius=1pt];
      \draw[fill=black] (.8*-1,0) circle [radius=1pt];
      \draw[fill=black] (.8*-.5,.8*-0.866) circle [radius=1pt];
      \draw[fill=black] (.8*.5,.8*-0.866) circle [radius=1pt];
    \end{tikzpicture}
\hspace{1em}
    \begin{tikzpicture}[scale=.8]
      \node at (1,.2)  {$\alpha_2$};
      \node at (.3,.9) {$\beta_2$};
      \node at (-.7,.75) {$\gamma_2$};
      \node at (-1,-.2) {$\delta_2$};
      \node at (-.3,-.9) {$\epsilon_2$};
      \node at (.7,-.75) {$\zeta_2$};
      
      \draw [-,scale = 1.3] (0,0) -- (1,0);
      \draw [-,scale = 1.3] (0,0) -- (.5,0.866);
      \draw [-,scale = 1.3] (0,0) -- (-.5,0.866);
      \draw [-,scale = 1.3] (0,0) -- (-1,0);
      \draw [-,scale = 1.3] (0,0) -- (-.5,-0.866);
      \draw [-,scale = 1.3] (0,0) -- (.5,-0.866);
      \draw [draw=none,fill=gray, opacity=0.2,scale = 1.3]
      (1,0) -- (.5,.866) -- (-.5,.866) -- (-1,0) -- (-.5,-.866) -- (.5, -.866) -- cycle;
      \draw[fill=black] (.8,0) circle [radius=1pt];
      \draw[fill=black] (.8*.5,.8*0.866) circle [radius=1pt];
      \draw[fill=black] (.8*-.5,.8*0.866) circle [radius=1pt];
      \draw[fill=black] (.8*-1,0) circle [radius=1pt];
      \draw[fill=black] (.8*-.5,.8*-0.866) circle [radius=1pt];
      \draw[fill=black] (.8*.5,.8*-0.866) circle [radius=1pt];
    \end{tikzpicture}
\hspace{1em}
\caption{Labelled simplicial fans of the genus $2$ surface as a tropical welded space}\label{hexfig}
\end{figure}
\end{example}

\section{Toric log symplectic manifolds}\label{sec3}

\begin{pp}

In analogy with the Poisson structure on the dual of any Lie algebra,
there is a canonical Poisson structure on the total space of the dual
bundle of any Lie algebroid~\cite{MR998124}. In the case of the
tangent bundle, this coincides with the canonical symplectic form on
the total space of the cotangent bundle. Given an action of $\RR^n$ on
$X$ generated by the vector fields $X_1,\ldots X_n$, let $A = \RR^n\ltimes
X$ be the corresponding abelian action Lie algebroid. If
$\{p_1,\ldots, p_n\}$ are standard vertical coordinates on
$N=\tot(A^*)$, the total space of the dual bundle of 
$A$, then the canonical Poisson structure is given by
\[
\sum_{k=1}^n \frac{\partial}{\partial p_k} \wedge X_k \in C^\infty(N,\wedge^2TN).
\]

If $X$ is taken to be a log affine manifold with free divisor $D$,
then $N \cong \RR^n\times X$ is also log affine for the pull back
divisor $Z=\RR^n\times D$.  The above Poisson structure is then
non-degenerate as a section of $\wedge^2(TN(-\log Z))$.  If
$\{\xi_1,\ldots, \xi_n\}$ is the basis of log forms dual to the basis
$\{X_1,\ldots, X_n\}$ of $TX(-\log D)$, then we may write the inverse of
the above Poisson structure as follows:
\begin{equation}
  \label{eq:23}
\omega = \sum_{k=1}^n \xi_k \wedge dp_k \in \Omega^2(N,\log Z).  
\end{equation}
This generalization of a symplectic structure is the focus of our
study and we characterize it as follows.
\begin{defn}
  Let $M$ be a manifold with free divisor $Z$ and corresponding
  Lie algebroid $TM(-\log Z)$.  A \textbf{log symplectic form}
  is a closed logarithmic 2-form
  \begin{equation}
    \label{eq:24}
    \omega \in \Omega^2(M,\log Z)
  \end{equation}
  which is non-degenerate in the sense that interior product defines 
  an isomorphism 
  \[
  \omega: TM(-\log Z)\to (TM(-\log Z))^*.
  \]
\end{defn}
\begin{remark}
  A log symplectic form may be viewed alternatively
  as a usual Poisson structure $\pi$ which degenerates along $Z$ in
  such a way that the top power $\pi^n$ is a reduced defining section
  for the free divisor $Z$.
\end{remark}

\end{pp}

\begin{pp}\label{quotastar}

The log symplectic form~\eqref{eq:23} is invariant under the global
$\RR^n$-action on $N$ by fiberwise translations.  We may therefore
quotient by the lattice $\Gamma = (2\pi \ZZ)^n\subset \RR^n$ and
obtain a trivial principal $T^n$-bundle $M$ equipped with a log
symplectic form invariant by the torus action, which we also denote by
$\Omega$. The bundle projection then induces a smooth submersion
\[
  \xymatrix{M\ar[r]^-{\mu} &  X }.
\]
Let $\Tt\cong \RR^n$ be the Lie algebra of $T^n$ and 
$\rho:\Tt\to C^\infty(M,TM)$ the infinitesimal principal action, 
which takes values in the sections of $TM(-\log\pi^*D)$. By contraction, 
we have a natural closed logarithmic 1-form with values in $\Tt^*$:
\[
i_{\rho} \omega \in \Omega^1(M,\log Z)\otimes\Tt^*,
\]
where $i_\rho \omega$ is the contraction of $\omega$ with $\rho$ in the first variable.
If $Z$ is of normal crossing type, we use the Mazzeo-Melrose 
decomposition~\ref{Mazzeo-Melrose} to express the logarithmic cohomology 
as 
\begin{equation}\label{mazmel}
H^1(M,\log Z) \cong H^1(M)\oplus \sum_i H^0(Z_i),
\end{equation}
where the sum is taken over the smooth components $Z_i$
of $Z$.

Recall that, as
a log affine manifold, $X$ carries a natural $\Tt^*$-valued
logarithmic form $\xi$ (see \eqref{eq:57}), 
  and so the map $\mu$ plays the role of a momentum map, in the sense
  that it satisfies
\begin{equation}
  \label{eq:21}
  i_{\rho} \omega = -\mu^* \xi.
\end{equation}

\begin{defn}\label{torlog}
  A \textbf{toric log symplectic manifold} of dimension $2n$ is a log
  symplectic manifold $(M, Z, \omega)$ of dimension $2n$, equipped
  with an effective action of the torus $T^n$ by symplectomorphisms,
  and a proper $T^n$-invariant smooth map $\mu:M\to X$ to a
  log affine manifold $(X,D,\xi)$, called a \textbf{tropical momentum map},
  satisfying condition~\eqref{eq:21}.  
  
  If $Z$ is a normal crossing divisor, we say that 
  the toric log symplectic manifold is \textbf{Hamiltonian}, when 
  the logarithmic cohomology class of $i_{\rho}\omega$ 
  has vanishing component in $H^1(M)\otimes\Tt^*$.
\end{defn}

\begin{remark}
  The most significant aspect of the above definition is that while each $i_{\rho(a)}\omega$ for $a\in\Tt$ is closed, we do not require it to be 
  \emph{exact}, that is, to be the differential of a Hamiltonian
  function. The reason for this is that such Hamiltonians would
  have logarithmic singularities along $Z$. Indeed, the component of $[i_{\rho(a)}\omega]$ in $H^0(Z_i)$ is given by its residue $r_i = \Res_{Z_i}(i_{\rho(a)}\omega)$, which is constant since $i_{\rho(a)}\omega$ is closed.  If $r_i\neq 0$, then in a neighbourhood of a point of $Z_i$ where the divisor has defining equation $x=0$, we have that $i_{\rho(a)}\omega - r_i d\log x$ is smooth, hence applying the Poincar\'e lemma,
\begin{equation}
i_{\rho(a)}\omega = d(r_i\log x + f)
\end{equation}  
for a smooth function $f$, yielding the claimed singular Hamiltonian $\mu_a = r_i\log x + f$.
  
  
  On the other hand, in the Hamiltonian toric case, we do require that the component of $[i_\rho\omega]$ in $H^1(M)\otimes\Tt^*$ vanishes, though the full logarithmic class, involving the residues described above, may not.  
  This condition specializes to the standard
  Hamiltonian assumption for usual symplectic manifolds. 
\end{remark}

\begin{remark}
One can show, using Proposition~\ref{prop: sympl uncut}, that the 
vanishing of the component of $[i_\rho\omega]$ in $H^1(M)\otimes\Tt^*$ is equivalent to the same condition on $[\xi]$, namely, that its Mazzeo-Melrose component in $H^1(X)\otimes\Tt^*$ vanishes.  This means that the Hamiltonian assumption above is equivalent to the assumption that $(X,D,\xi)$ has trivial affine monodromy. 
\end{remark}

\end{pp}

\subsection{Trivial principal case}
\begin{pp}
  In Paragraph~\ref{quotastar} we described the simplest example of a
  toric log symplectic manifold: a trivial principal
  $T^n$-bundle over a log affine manifold.
\begin{prop}
  Let $(X,D,\xi)$ be a log affine manifold, and $M$ the trivial
  principal $T^n$-bundle over $X$.  The log symplectic
  form~\eqref{eq:23} then defines a toric log symplectic
  structure on $M$ with momentum map $\mu$ given by the bundle
  projection and degeneracy locus $Z=\mu^{-1}(D)$.
\end{prop}

\end{pp}

\begin{pp}
This example may be deformed by adding a \emph{magnetic term}. That is, if $\omega$ is the canonical symplectic
form~\eqref{eq:23}, then
\[
\omega + \mu^* B
\]
is also a log symplectic form, for any closed form $B\in\Omega^2(X,\log D)$. 
By applying a modified version of Moser's method as in \cite[{Theorem 38}]{MR3250302}, one can
show that only the logarithmic cohomology class of $B$ is relevant to
the resulting symplectic structure up to equivariant
symplectomorphism, yielding the following result.

\begin{theorem}
  Fix the log affine manifold $(X,D,\xi)$ and the trivial principal 
$T^n$-bundle $\mu:M\to X$ as above, and let $\omega_0$ be the canonical
  log symplectic form~\eqref{eq:23}.  If $\omega$ is another log
  symplectic form such that $(M,\omega,\mu)$ is toric log symplectic,
  then the difference $\omega-\omega_0$ is basic
  and closed. Furthermore, the map
  \[
  \omega\mapsto [\omega-\omega_0]\in H^2(X,\log D)
  \]
  induces a bijection between equivalence classes of toric 
  log symplectic structures with fixed momentum map $\mu$ and the
  vector space $H^2(X,\log D)$.
\end{theorem}

\begin{remark}
  The usual notions of invariant and basic forms on a principal bundle
  carry over to the logarithmic context above, because of the exact sequence
  relating logarithmic vector fields on the domain and codomain of $\pi$:
  \begin{equation}
    \label{eq:27}
    \xymatrix{0\ar[r] & T_{M/X}\ar[r] & T_M(-\log \mu^*D)\ar[r]^{\mu_*} & \mu^*T_X(-\log D)\ar[r] & 0},   
    \end{equation}
where $T_{M/X}$ is the vertical tangent bundle of the principal bundle $M$. 
  This is an example of an algebroid submersion (see \cite{GLP1}) and
  gives rise to an injective cochain homomorphism from the log de Rham complex of
  $(X,D)$ to the log de Rham complex of $(M,\mu^*D)$.
\end{remark}

\begin{remark}
  By an argument similar to that of Mazzeo-Melrose~\cite{MR1734130}
  (see Appendix~\ref{Mazzeo-Melrose}), one can show that for $D\subset
  X$ a normal crossing divisor, the logarithmic cohomology groups may
  be expressed in terms of usual de Rham cohomology:
  \[
  H^2(X,\log D) = H^2(X) \oplus \sum H^1(D_i)  \oplus \sum H^0(D_i\cap D_j),
  \]
  where we sum over the components $D_i$ of $D$ and their pairwise
  intersections.

  For example, the log affine structure on $S^2$ described in
  Example~\ref{spherethreecircles} has a divisor with three circular
  components intersecting in six points; as a result there is a
  10-dimensional space of toric log symplectic structures
  on the trivial $T^2$-bundle.
\end{remark}
\end{pp}

\subsection{Nontrivial principal case}

\begin{pp}
We now show that certain nontrivial principal $T^n$-bundles over the
log affine manifold $(X,D,\xi)$ admit toric log symplectic
structures.

\begin{defn}
  Let $\pi:M\to X$ be a principal $T^n$-bundle over the log affine
  $n$-manifold $(X,D,\xi)$, with real Chern classes 
  \(
  c_1(M) = (c_1^{1},\ldots c_1^n)\in H^2(X,\RR)\otimes \Tt.
  \)
The \textbf{obstruction class} of $M$ is then given by 
  \begin{equation}
    \label{eq:25}
    \mathrm{Tr}(c_1(M)\wedge [\xi]) = \sum_{k=1}^n c_1^k \wedge [\xi_k] \in H^3(X,\log D).
  \end{equation}
\end{defn}

\begin{theorem}\label{nontrivprinc}
With $M$ and $X$ as above, there exists a log symplectic form $\omega$ making $(M,\omega,\pi)$
  a toric log symplectic manifold if and only if the obstruction class~\eqref{eq:25} vanishes. 
  
  Under this condition, the space of equivalence classes of log
  symplectic forms making $(M,\omega,\pi)$ toric log symplectic is 
  an affine space modeled on the vector space $H^2(X,\log D)$.
\end{theorem}
\begin{proof}
  Let $(M,\omega,\pi)$ be as above. Choosing a principal connection
  $\theta\in \Omega^1(M)\otimes \Tt$, we may write 
  \begin{equation}\label{logsympformprinc}
  \omega = \sum_{k=1}^n \theta^k \wedge \alpha_k + \beta,
  \end{equation}
  where $\alpha_k, \beta$ are invariant and basic logarithmic forms.
  The momentum map condition~\eqref{eq:21} implies that $\alpha_k = \xi_k$.
  Taking derivatives, we obtain 
  \[
  \sum_{k=1}^n F^k \wedge \xi_k + d\beta = 0,
  \]
  where $[F^k] = c_1^k$, so that~\eqref{eq:25} vanishes, as required.

  Conversely, choose local trivializations for the principal 
  $T^n$-bundle $M$ over an open cover $\{U_i\}$ of $X$, and let
  \[
  \varphi_{ij}: (U_i\times T^n)|_{U_{i}\cap U_j}\to (U_j\times T^n)|_{U_{i}\cap U_j}
  \]
  be the transition isomorphisms satisfying the usual cocycle
  condition.  On the $k^\mathrm{th}$ circle factor, $\varphi_{ij}$ acts
  by multiplication by $g^{k}_{ij}:U_i\cap U_j\to S^1$.  Also, choose
  a principal connection, defined by connection forms
  $A^{k}_i\in\Omega^1(U_i,\log D)$ satisfying the cocycle condition
$i(A^{k}_i - A^{k}_j) = d\log g^{k}_{ij}$.

  We first equip each $U_i \times T^n$ with the standard
  log symplectic form given in~\eqref{eq:23}, namely,
  \[
  \tilde\omega_i = \sum_{k=1}^n \xi_k\wedge dp_k.
  \]
  If we write $g^{k}_{ij} = \exp(i\tau_{ij}^{k})$, then we have 
  \begin{equation}
    \label{eq:28}
    \varphi_{ij}^* \tilde\omega_j - \tilde\omega_i = \sum_{k=1}^n \xi_k \wedge d\tau_{ij}^{k}
    = \sum_{k=1}^n \xi_k \wedge (A^{k}_i - A^{k}_j).
  \end{equation}
  If the class~\eqref{eq:25} vanishes, then there
  exists $B\in \Omega^2(X,\log D)$ such that, over $U_i$, 
  \begin{equation}
    \label{eq:29}
    dB = \sum_{k=1}^n F^k \wedge \xi_k = \sum_{k=1}^n dA^{k}_i\wedge\xi_k.
  \end{equation}
Combining~\eqref{eq:28} and~\eqref{eq:29}, we see that the modified symplectic forms
\[
\omega_i = \tilde\omega_i + \mu^*(B|_{U_i} + \sum_{k=1}^n \xi_k\wedge A_i^k) 
\]
define a global form $\omega$ on $M$ rendering
$(M,\omega,\mu)$ toric Hamiltonian, as required.

Due to the choice of the 2-form $B$, we see that the symplectic form
is only uniquely determined modulo closed logarithmic 2-forms on $X$.
Applying the Moser method, which uses the properness of $\mu$, we
therefore obtain a free and transitive action of $H^2(X,\log D)$ on
the set of equivalence classes of toric Hamiltonian log symplectic
structures $(M,\omega,\mu)$, as needed.
\end{proof}

\begin{remark}
  For $D\subset X$ a normal crossing divisor, we have the following
  expression for the third logarithmic cohomology:
  \[
  H^3(X,\log D) = H^3(X) \oplus \sum H^2(D_i) \oplus \sum H^1(D_i\cap
  D_j)\oplus \sum H^0(D_i\cap D_j\cap D_k),
  \]
  where we sum over double and triple intersections of the components
  $D_i$ of $D$.  
  If we also have $\dim X = 2$, then this group
  vanishes. 
\end{remark}

\begin{example}
  For the log affine structure on $S^2$ described in
  Example~\ref{spherethreecircles}, we have a 10-dimensional affine
  moduli space of toric log symplectic structures on any
  principal $T^2$-bundle over $S^2$, including, for example, the Hopf
  manifold $S^3\times S^1$.
\end{example}

\begin{example}
  The log affine structure in Example~\ref{hex} has a divisor with six
  components intersecting in six $\RR^2$ fixed points. As a result, we
  obtain a 13-dimensional affine moduli space of toric log
  symplectic structures on any principal $T^2$-bundle over the
  orientable genus 2 surface.
\end{example}

\end{pp}

\begin{pp}

  Theorem~\ref{nontrivprinc} provides a classification of toric log symplectic
   structures on a principal $T^n$-bundle which satisfies
  condition~\eqref{eq:25}.  Note, however, that the notion of equivalence is 
  that of equivariant symplectomorphisms $\psi:M\to M$, in the sense that 
  $\mu\circ\psi = \mu$.  In some cases, there are additional symmetries which 
  do not commute with $\mu$; we describe them below.
  
  The main observation is that while toric log symplectic forms on such a
  principal bundle may be deformed by closed basic forms, the map on
  cohomology groups has a kernel:  we have the exact Gysin sequence 
  \begin{equation}
    \label{eq:30}
    \xymatrix@C=4em{H^0(X,\log D)\otimes \Tt^* \ar[r]^-{c_1(M)} & H^2(X,\log D)\ar[r]^{\mu^*} & H^2(M,\log \pi^* D)}.
  \end{equation}
Thus, we expect that 2-forms in the kernel of $\mu^*$
  would act trivially on the symplectomorphism class of $(M,\omega)$.
  
  Indeed, let $v = (v_1,\ldots, v_n) \in \Tt^*$ and consider the  class
  \[
  c_1(M)\cdot v = \sum_{k=1}^n v_k c_1^k  \in H^2(X,\log D).
  \]
  Choosing a connection $\theta$ on $M$ as in the proof of
  Theorem~\ref{nontrivprinc}, the above class gives rise to a 1-parameter 
  family of toric log symplectic structures 
  \[
  \omega_t = \omega + t \sum_{k=1}^n v_k F^k.
  \]
  The Moser argument then indicates that the flow of
  $X_t=\omega_t^{-1}(v\cdot \theta)$ trivializes the above family.
  This vector field is invariant and projects to the vector field
  $\rho(v)$ on the log affine manifold $X$. Therefore, $X_t$  is complete
  (allowing the Moser argument to proceed) if and only if $\rho(v)$ is complete,
  yielding the following result.

  \begin{theorem}
    If $\omega,\omega'$ are two toric log symplectic
    structures on the principal $T^n$-bundle $(M,\pi)$ over the
    complete log affine manifold $(X,D,\xi)$, and if
    $[\omega'-\omega]$ is of the form $c_1(M)\cdot v$ for $v\in\Tt^*$,
    then there is a symplectomorphism $\psi:(M,\omega)\to(M,\omega')$
    such that $\mu \circ\psi = T_v \circ \mu$, where $T_v: X\to X$ is
    the automorphism given by translation by $v$.
  \end{theorem}

\end{pp}

\section{Polytopes and symplectic cutting}\label{cuts}
In this section, $(X,D,\xi)$ denotes a fixed complete log affine
$n$-manifold, possibly with corners, and we assume that it is a tropical welded space of the type
constructed in Section~\ref{weld}, so that all of its strata
(i.e., $\RR^n$-orbits) are affine spaces and $D$ is of normal crossing
type.  We use the notation $U=\RR^n$ as before, for convenience.  If 
the log affine manifold is the codomain of a tropical momentum map, we have $U = \Tt^*$.

\subsection{Log affine polytopes}

\begin{pp} Each stratum of $X$ is an affine space, where there is a
  clear notion of affine hyperplane. We extend this notion
  as follows.

\begin{defn}
  An \textbf{affine linear hypersurface} in the log affine manifold
  $(X,D,\xi)$ is a connected and embedded hypersurface in $X$, not
  contained in $D$, and invariant under the action of a linear
  hyperplane in $U$.
\end{defn}

\begin{defn}
  A \textbf{log affine polytope} is a closed equidimensional
  submanifold with corners of $X$, each of whose codimension 1 boundary
  strata has closure which is either:
  \begin{itemize}
  \item contained in $D$, and called a \textbf{singular face}, or
  \item contained in an affine linear hypersurface.  Such a stratum
    may intersect $D$, in which case we require it to intersect all 
    possible intersections of components of $D$ transversely (so its union with $D$ is
    also of normal crossing type), and we call the stratum a \textbf{log
      face}.  If the stratum does not intersect $D$, we call it an
    \textbf{interior face}.
  \end{itemize}
  The polytope is called \textbf{convex} when its intersection with each
  component of $X\setminus D$ is convex.
\end{defn}
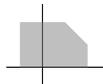
\begin{figure}[H]
    \centering
    \begin{tikzpicture}[scale=1.2]
\draw (-.4,0) -- (.7,0);
\draw (0,-.2) -- (0,.7);
\draw [draw=none,fill=gray, opacity=0.5]
      (-.25,0) -- (-.25,.5) -- (.25,.5) -- (.5,.25)-- (.5,0) -- cycle;
    \end{tikzpicture}
\caption{Convex log affine polytope with one singular, one interior, and three log faces}
\end{figure}
\end{pp}

\begin{pp}\label{pulbk}
  Suppose the hypersurface $H\subset X$ is not contained in $D$, but
  is affine linear for the hyperplane $U_H\subset U$.  If $H\cup
  D$ is normal crossing, then $H$ inherits a log affine structure,
  with divisor $D_H=D\cap H$ and trivialization
  $\xi_H\in\Omega^1(H,\log D_H)\otimes U_H$ given by the pull back of $\xi$ to $H$. As a result
  of this, we see that the closure of any log face contained in $H$ is
  itself a log affine polytope.

  Similarly, if $C$ is a component of the degeneracy divisor $D$, then
  it inherits a log affine structure, with divisor $D_C =
  \overline{(D\setminus C)} \cap C$ and trivialization
  $\xi_C\in\Omega^1(C,\log D_C)\otimes U_C$ given by the
  isomorphism of extensions
  \[
  \xymatrix@R=1em{ \underline{\RR}\ar[r] & TX(-\log D)|_C \ar[r] & TC(-\log D_C) \\
 \mathfrak{s}_C\ar[r]\ar[u] & U \ar[r]\ar[u]_{\xi} & U_C\ar[u]_{\xi_C}}
  \]
  where $\mathfrak{s}_C$ is the 1-dimensional stabilizer of a generic
  point on $C$.  Therefore, the closure of any singular face
  contained in $C$ is also a log affine polytope.  

  \begin{prop}
    The closure of any stratum of a log affine polytope is itself a
    log affine polytope.
  \end{prop}
\end{pp}

\begin{pp}
  Any interior face $F$ of a log affine polytope is an affine polytope
  in the usual sense; if it has dimension $k$, its associated
  1-form $\xi\in \Omega^1(F)\otimes \RR^k$ defines a volume form $\det
  \xi = \xi_1\wedge \cdots \wedge \xi_k$. Thus, any compact 
  interior face has a well-defined volume in $\RR_+$. 

  Even logarithmic volume forms may have well-defined integrals over
  compact manifolds, by extracting the Cauchy principal value:
  combining the results of~\cite{Torres01062004} with Fubini's theorem
  (and using the normal crossing assumption on $D$), we see that any
  compact and oriented log affine polytope has a well-defined volume, as long as
  none of its faces is contained in $D$.
\begin{defn} \label{def: regvol}
  The \textbf{regularized volume} of a compact and oriented log affine polytope
  $\Delta\subset (X,D,\xi)$ without singular faces is the real number
  defined by the iterated Cauchy principal value
  \[
  \mathrm{Vol}(\Delta) = PV \int_{\Delta} \det \xi.
  \]
 \end{defn}
\end{pp}

\begin{example} \label{ex: regvol}
Consider the log affine manifold $(\RR^2, D, \xi)$, where $D=\{xy=0\}$  and $\xi = (x^{-1}dx, y^{-1}dy) \in \Omega^1(\RR^2, \log D) \otimes \RR^2$.   
Let $\Delta\subset\RR^2$ be the square log affine polytope with vertices $(2, 2)$, $(2, -1)$, $(-1, -1)$ and $(-1, 2)$. 
The regularized volume of $\Delta$ is then 
$$
	PV\int_\Delta \det \xi = \left(PV\int_{-1}^2 x^{-1} dx\right)\left(  PV \int_{-1}^2 y^{-1} dy\right) = (\log 2)^2.
$$
\end{example}

\begin{defn}
  An \textbf{affine linear function} on the log affine manifold $(X,D,\xi)$ is
  a function $f:X\to \RR$ such that there exists a covector $a\in
  (\RR^n)^*$ (called the \emph{linear part} of $f$) such that for all
  $u\in\RR^n$,
  \begin{equation}
    \label{eq:41}
    L_{\rho(u)} f = a(u),
  \end{equation}
	where $L_{\rho(u)} f$ is the Lie derivative of $f$ along the vector field $\rho(u)$.

In particular, when $X$ is complete, $f:X\to \RR$ is affine linear if $f(x + u) = f(x) + a(u)$ for
  all $x\in X$ and translations $u\in\RR^n$.
\end{defn}

\begin{pp}
  let $\Delta\subset X$ be a log affine polytope and $F\subset
  \Delta$ any nonsingular face.  Then there is an affine linear
  function $f$ defined in a tubular neighborhood $V_F$ of $F$
  such that $f|_F = 0$ and $f\geq 0$ on $\Delta\cap V_F$.  We call
  this a \emph{boundary defining function} for the face $F$.
\begin{lemma}
  Let $F$ be a log face with boundary defining function $f$.  If $F$
  intersects a component $D_i$ of the divisor $D$ with associated
  residue $v_i$, then the linear part of $f$ satisfies $a(v_i)=0$.
\end{lemma}
\begin{proof}
  If $p\in F\cap D_i$, then it is fixed by the action of $v_i$, 
  and so $L_{\rho(v_i)} f = a(v_i) = 0$, as needed.
\end{proof}

\begin{defn}
  An \textbf{elementary} log affine polytope is a convex log affine
  polytope $\Delta$ with $\Delta\setminus D$ connected. 
\end{defn}

\begin{lemma}
  Let $\Delta$ be an elementary log affine polytope, $F$ a
  nonsingular face of $\Delta$, and  $f$ a boundary defining
  function for $F$ with linear part $a$.  If a component $D_i$ of the divisor $D$ with
  associated residue $v_i$ meets $\Delta$ but does not intersect $F$,
  then $a(v_i) < 0$.
\end{lemma}
\begin{proof}
  By convexity, each point $p\in\Delta\setminus D$ lies on a unique
  straight line in the direction $v_i$ whose closure meets
  $D_i$ (at infinity in the $-v_i$ direction).  Since $f$ is positive
  on the interior of $\Delta$, it follows that $L_{\rho(v_i)}f <0$.
\end{proof}

\end{pp}  

\begin{pp}
Let $\Delta$ be an elementary log affine polytope in $(X, D, \xi)$. Since $\Delta\setminus D$ is connected, we might as well assume $X\setminus D$ has only one component and $D = \partial X$. By Proposition~\ref{prop: semi-local} and Theorem~\ref{thm: weld}, the log affine manifold $(X, \partial X, \xi)$ is classified by its associated labelled simplicial fan $\Sigma$ in a real vector space $U$. Uniquely defined up to a positive scalar,
  the boundary defining function $f_i$ for each of the nonsingular
  faces $F_i, i\in I$, of $\Delta$ has linear part $a_i$ which
  determines a half-space
  \[
  H_i = \{ u\in U : a_i(u) > 0\}.
  \]
  By the above Lemmas, $H_i$ does not intersect $\Sigma$.
  Compactness of $\Delta$ may be expressed in terms of these half
  spaces, as follows.
  \begin{prop}
    The elementary log affine polytope $\Delta$ is compact if and only if 
    \[
    \Sigma \sqcup \left(\cup_{i\in I} H_{i}\right) = U.
    \]
  \end{prop}
\begin{example}
   Consider a square $s = (-1, 1]^2$ labeled with the vectors $(1, 0)$ and $(1, 1)$. The square $s$ defines a labeled simplicial fan $\Sigma = (\sF, \sC)$ in $\RR^2$ as in Definition~\ref{blocks}, where $\sF =\{(1, 0), (1, 1)\}$ and $\sC$ contain all subsets of $\sF$. Let $X_s$ be the associated log affine manifold as in Definition~\ref{def: X_c}, see also ii) of Example~\ref{ex: X_c}. Let $\Delta\subset X_s$ be the log affine
  polytope given by the closure of $\{f_1\geq 0\}\cap \{f_2\geq 0\}$, for 
  \[
  f_1(x,y) = - y,\qquad f_2(x,y) = -x+y.
  \]
  Then $\Delta$ is a compact convex log affine manifold with two
  singular faces and two log faces.  Figure~\ref{compdelt} shows the
  union of the fan defining $X_s$ with the pair of half-planes defined by the
  nonsingular faces of $\Delta\subset X$.
 \begin{figure}[H]
    \centering
    \begin{tikzpicture}
      \draw [draw=none,fill=gray, opacity=0.3]
      (0,0) -- (1,0) -- (1,1) -- cycle;
      \draw (0,0) -- (1,0);
      \draw (0,0) -- (1,1);
      \draw [dashed] (-1,-1) -- (0,0);
      \draw [dashed] (-1,0) -- (0,0);
      \draw [draw=none, pattern=north east lines, opacity = 0.2]
      (-1,-1)--(-1,1)--(1,1) -- cycle;
      \draw [draw=none, pattern=horizontal lines, opacity = 0.2]
      (-1,0)--(1,0)--(1,-1) -- (-1,-1)--cycle;
      \node at (-.6,.6) {\small $H_2$};
      \node at (0.4,-.6) {\small $H_1$};
      
    \end{tikzpicture}
    \caption{Fan and half-spaces defined by the compact polytope $\Delta$}\label{compdelt}
  \end{figure}
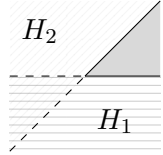

\end{example}

\end{pp}

\subsection{Symplectic cuts}\label{scut}

\begin{pp}

  Let $\Delta\subset X$ be a log affine polytope and $F\subset
  \Delta$ a nonsingular face with boundary defining function $f$.
  Let $(M,Z,\omega,\pi)$ be a principal toric Hamiltonian space over
  $X$, as in Theorem~\ref{nontrivprinc}. If the linear part of $f$ is
  a primitive integral vector $a\in\Tt = \RR^n$, then the Hamiltonian
  function $\pi^*f$ generates the action of a circle subgroup $S^1_a
  \subset T^n$ on $M$ and we may define a log symplectic cut of $M$
  along $\pi^{-1}(F)$, generalizing the usual symplectic
  cut~\cite{MR1338784} as follows. Note that while $f$ is only defined
  in a neighborhood of $F$, this is irrelevant to the symplectic cut
  procedure, in which we only modify $M$ near $\pi^{-1}(F)$.  We
  assume for simplicity that $\pi^*f$ is defined on all of $M$.

  \begin{theorem}
    Let $(M,Z,\omega)$ be a log symplectic manifold and $f:M\to
    \RR$ a Hamiltonian function generating an $S^1$-action.
    Suppose that the $S^1$-action is free along the orbit $f^{-1}(0)$.
    
    Then $M\times \CC$, equipped with the log symplectic form $\omega
    + idz\wedge d\bar z$, has an antidiagonal $S^1$-action whose
    Hamiltonian $\mu = f - |z|^2$ gives rise to a smooth log symplectic reduction 
    \begin{equation}
      \label{eq:45}
      M_f = \mu^{-1}(0)/S^1
    \end{equation}
    called the \textbf{log symplectic cut}.  The Hamiltonian $f$
    descends to a function $\tilde f$ on $M_f$, defining a residual
    $S^1$-action which now fixes the log symplectic submanifold
    $S=\tilde f^{-1}(0)$.  As in the usual symplectic cut, we 
    identify $S$ with the log symplectic reduction $f^{-1}(0)/S^1$,
    whereas $M_f\setminus S$ is isomorphic to $M\setminus f^{-1}(0)$.
  \end{theorem}
  \begin{proof}
    Since the $S^1$-action is free along $f^{-1}(0)$, the generating
    vector field $\del_\theta$ is nowhere vanishing there, and so by
    the non-degeneracy of $\omega$ along this locus, $ df =
    -i_{\del_\theta}\omega$ is a nowhere vanishing log 1-form. Hence 
    \[
    df: TM(-\log Z)|_{f^{-1}(0)} \to T_0\RR
    \]
    is surjective, verifying the transversality condition of
    Proposition~\ref{logred}.  This immediately implies that the
    transversality condition for $\mu$ also holds.  The same freeness
    assumption also implies that the antidiagonal $S^1$-action on
    $\mu^{-1}(0)$ is free, giving a smooth log symplectic quotient by
    Proposition~\ref{logred}.  The remaining statements follow from the same
    arguments given for usual symplectic cutting.
  \end{proof}

\end{pp}

\begin{pp}
  The above construction allows us, just as in the usual Delzant
  theory, to iteratively apply symplectic cuts along the nonsingular
  faces of $\Delta$, under the usual Delzant condition that whenever
  $k$ such faces meet, we may choose boundary defining functions whose
  linear parts define an integral basis for a maximal rank $k$
  sublattice of $\ZZ^n$.  
\begin{defn}\label{delzantcondition}
  The log affine polytope $\Delta$ satisfies the \textbf{Delzant
    condition} when, for any point $p\in\Delta$, the collection of $k$
  nonsingular faces meeting $p$ defines an integral basis for a
  maximal rank $k$ sublattice of $\ZZ^n\subset \Tt$. In this case, 
  $\Delta$ is called a \textbf{Delzant log affine polytope}.
\end{defn}

\begin{corollary}\label{sympcutcon}
  Given any log affine polytope $\Delta$ in the log affine manifold
  $(X,D,\xi)$, and given Chern classes $(c_1^1,\ldots c_1^n)\in
  H^2(X,\RR)\otimes\Tt$ such that the obstruction class defined
  in~\eqref{eq:25} vanishes, we may
  construct a corresponding principal toric Hamiltonian log symplectic
  manifold $(\widetilde M, \widetilde Z, \widetilde \omega)$ with $Z$ of normal
  crossing type and momentum map to $X$ as in
  Theorem~\ref{nontrivprinc}.

  If $\Delta$ satisfies the Delzant condition, we may apply symplectic
  cuts to $\widetilde M$ along all of the nonsingular faces of $\Delta$ to
  produce another toric Hamiltonian log symplectic manifold
  $(M,Z,\omega)$, also with $Z$ of normal crossing type, with an
  identical momentum map image.  The toric orbit type stratification
  of $(M,Z)$ then coincides with the stratification of $\Delta$, via the
  momentum map.

  If, in addition, $\Delta$ has no singular faces, then $(M,Z,\omega)$ is a 
  smooth log symplectic manifold without boundary. 
\end{corollary}

\begin{example}\label{gen1}
  On the log affine manifold $(X,D,\xi)$ constructed in
  Example~\ref{hex}, we choose an affine linear function $h$ whose
  linear part vanishes on the vectors $\alpha=a=-\delta=-d$ and is
  positive on $\epsilon=e$ and $\zeta=f$.  The function $h$ defines a
  compact log affine polytope $\Delta$ with a single log face and no
  singular faces.  In Figure~\ref{fig:halfspaces}, we show the fans
  corresponding to the strata of $X$ which meet $\Delta$, as well as
  the half-spaces $H_1,\ldots, H_4$ defined by the linear parts of $h$.  In Figure~\ref{fig:genus2}, we show the orientable
  genus $2$ surface $X$ together with its decomposition into four
  hexagonal components; we see that the log affine polytope $\Delta$ is 
  a 2-dimensional 
  submanifold of genus $1$ with a single log face.

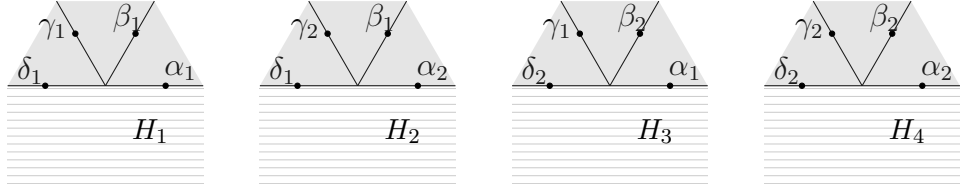
\begin{figure}[H]
    \centering
    \begin{tikzpicture}
      \node at (1,.2)  {$\alpha_1$};
      \node at (.3,.9) {$\beta_1$};
      \node at (-.7,.75) {$\gamma_1$};
      \node at (-1,.2) {$\delta_1$};
      
      \draw [-,scale = 1.3] (0,0) -- (1,0);
      \draw [-,scale = 1.3] (0,0) -- (.5,0.866);
      \draw [-,scale = 1.3] (0,0) -- (-.5,0.866);
      \draw [-,scale = 1.3] (0,0) -- (-1,0);

      \draw [draw=none, pattern=horizontal lines, opacity = 0.2,scale=1.3]
      (1,0) -- (1,-1) -- (-1,-1) -- (-1,0) -- cycle;
      \node at (.6,-.6) {\small $H_1$};
 
      \draw [draw=none,fill=gray, opacity=0.2,scale = 1.3]
      (1,0) -- (.5,.866) -- (-.5,.866) -- (-1,0) -- cycle;
             \draw[fill=black] (.8,0) circle [radius=1pt];
      \draw[fill=black] (.8*.5,.8*0.866) circle [radius=1pt];
      \draw[fill=black] (.8*-.5,.8*0.866) circle [radius=1pt];
      \draw[fill=black] (.8*-1,0) circle [radius=1pt];

    \end{tikzpicture}
\hspace{1em}
    \begin{tikzpicture}
      \node at (1,.2)  {$\alpha_2$};
      \node at (.3,.9) {$\beta_1$};
      \node at (-.7,.75) {$\gamma_2$};
      \node at (-1,.2) {$\delta_1$};
      
      \draw [-,scale = 1.3] (0,0) -- (1,0);
      \draw [-,scale = 1.3] (0,0) -- (.5,0.866);
      \draw [-,scale = 1.3] (0,0) -- (-.5,0.866);
      \draw [-,scale = 1.3] (0,0) -- (-1,0);
      \draw [draw=none, pattern=horizontal lines, opacity = 0.2,scale=1.3]
      (1,0) -- (1,-1) -- (-1,-1) -- (-1,0) -- cycle;
      \node at (.6,-.6) {\small $H_2$};
      
      \draw [draw=none,fill=gray, opacity=0.2,scale = 1.3]
      (1,0) -- (.5,.866) -- (-.5,.866) -- (-1,0) -- cycle;
        \draw[fill=black] (.8,0) circle [radius=1pt];
      \draw[fill=black] (.8*.5,.8*0.866) circle [radius=1pt];
      \draw[fill=black] (.8*-.5,.8*0.866) circle [radius=1pt];
      \draw[fill=black] (.8*-1,0) circle [radius=1pt];

    \end{tikzpicture}
\hspace{1em}
    \begin{tikzpicture}
      \node at (1,.2)  {$\alpha_1$};
      \node at (.3,.9) {$\beta_2$};
      \node at (-.7,.75) {$\gamma_1$};
      \node at (-1,.2) {$\delta_2$};
      
      \draw [-,scale = 1.3] (0,0) -- (1,0);
      \draw [-,scale = 1.3] (0,0) -- (.5,0.866);
      \draw [-,scale = 1.3] (0,0) -- (-.5,0.866);
      \draw [-,scale = 1.3] (0,0) -- (-1,0);
      
      \draw [draw=none, pattern=horizontal lines, opacity = 0.2,scale=1.3]
      (1,0) -- (1,-1) -- (-1,-1) -- (-1,0) -- cycle;
      \node at (.6,-.6) {\small $H_3$};

      \draw [draw=none,fill=gray, opacity=0.2,scale = 1.3]
      (1,0) -- (.5,.866) -- (-.5,.866) -- (-1,0) -- cycle;
        \draw[fill=black] (.8,0) circle [radius=1pt];
      \draw[fill=black] (.8*.5,.8*0.866) circle [radius=1pt];
      \draw[fill=black] (.8*-.5,.8*0.866) circle [radius=1pt];
      \draw[fill=black] (.8*-1,0) circle [radius=1pt];

          \end{tikzpicture}
\hspace{1em}
    \begin{tikzpicture}
      \node at (1,.2)  {$\alpha_2$};
      \node at (.3,.9) {$\beta_2$};
      \node at (-.7,.75) {$\gamma_2$};
      \node at (-1,.2) {$\delta_2$};
      
      \draw [-,scale = 1.3] (0,0) -- (1,0);
      \draw [-,scale = 1.3] (0,0) -- (.5,0.866);
      \draw [-,scale = 1.3] (0,0) -- (-.5,0.866);
      \draw [-,scale = 1.3] (0,0) -- (-1,0);
      
      \draw [draw=none, pattern=horizontal lines, opacity = 0.2,scale=1.3]
      (1,0) -- (1,-1) -- (-1,-1) -- (-1,0) -- cycle;
      \node at (.6,-.6) {\small $H_4$};

      \draw [draw=none,fill=gray, opacity=0.2,scale = 1.3]
      (1,0) -- (.5,.866) -- (-.5,.866) -- (-1,0) -- cycle;
        \draw[fill=black] (.8,0) circle [radius=1pt];
      \draw[fill=black] (.8*.5,.8*0.866) circle [radius=1pt];
      \draw[fill=black] (.8*-.5,.8*0.866) circle [radius=1pt];
      \draw[fill=black] (.8*-1,0) circle [radius=1pt];

    \end{tikzpicture}
\hspace{1em}
\caption{Fan and half-spaces defined by polytope $\Delta$ in Example~\ref{gen1}}
\label{fig:halfspaces}
\end{figure}

\begin{figure}[H]
\centering
\includegraphics[scale=.25]{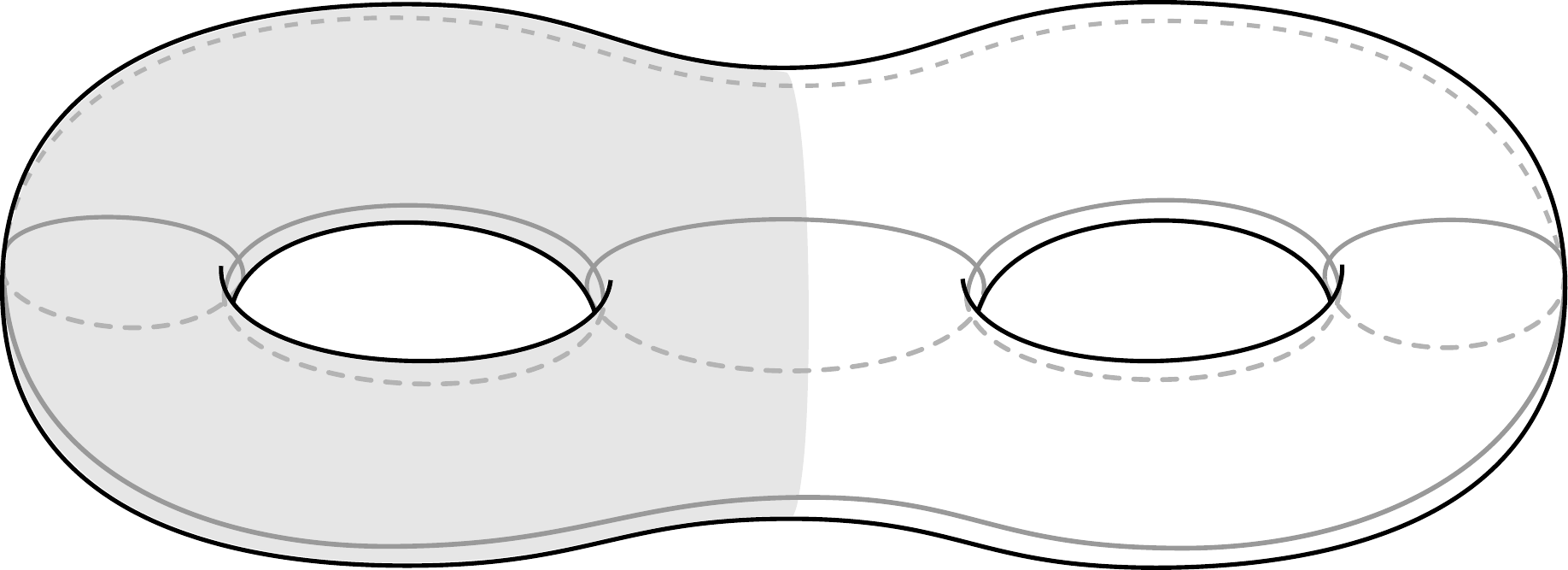}
\caption{Log affine polytope of genus 1 with single log face, sitting
  in a log affine manifold of genus 2 as in Example~\ref{hex}. }
\label{fig:genus2}
\end{figure}

This example demonstrates that log affine polytopes may have
nontrivial topology and so toric Hamiltonian log symplectic manifolds
need not be equivariantly diffeomorphic to usual toric symplectic
manifolds, or even quasitoric manifolds.  In fact, we may take the
trivial principal $T^2$-bundle over the polytope $\Delta$, i.e., the shaded area in Figure~\ref{fig:genus2}, and perform a symplectic cut along the unique log face, i.e., the boundary of $\Delta$. By Proposition 1 in~\cite{MR0375320}, we conclude that the resulting manifold is
diffeomorphic to $S^1 \times ((S^1\times S^2)\# (S^1\times S^2))$, and
by Lemma 2.3 of~\cite{MR1707526}, this manifold, although it is almost
complex, does not admit any symplectic structure, by an argument
involving the vanishing of Seiberg-Witten invariants.  Nevertheless, it admits
a log symplectic structure by our construction.
\end{example}

\end{pp}

\section{Delzant correspondence}\label{sec5}

\begin{pp}
  In this section, we establish an analogue of the Delzant
  correspondence for a class of toric Hamiltonian log symplectic
  manifolds.  The correspondence differs from Delzant's 
  in several important ways. First, the analog of the Delzant
  polytope is a log affine polytope $\Delta$ sitting in a tropical welded space; it need not be contractible. Second,
  the $n$-dimensional polytope $\Delta$ is decorated by an $n$-tuple
  of integral classes in $H^2(\Delta)$. Finally, with the above data
  fixed, the equivariant isomorphism class of the log symplectic form
  is unique only up to addition of a magnetic term coming from the
  logarithmic cohomology of $\Delta$.
\end{pp}

\begin{pp}
  We prove an analogue of the Delzant classification for compact orientable
  log symplectic manifolds $(M,Z,\omega)$, under the assumption that $Z$ is 
  a real divisor of normal crossing type and such that $\del M\subset Z$.
  We assume that there is an effective action of $T^n$ preserving the 
  structure which is \emph{Lagrangian}, meaning that the infinitesimal
  action $\rho:\Tt\to C^\infty(M,TM)$ satisfies $\omega(\rho(a),\rho(b))=0$ 
  for all $a,b\in\Tt$.  

	\begin{theorem} \label{delzant}
	 
		Let $(M,Z,\omega)$ be an orientable compact and connected log symplectic manifold 
		with $Z$ of normal crossing type and $\del M\subset Z$,
		equipped	 with an effective Lagrangian $T^n$-action.  
		Then the following hold:
	\begin{enumerate}
	\item[1.]		
		The quotient $M/T^n$ defines a compact and connected Delzant log affine polytope $(\Delta,D,\xi)$ with $D$ of normal crossing type.  Furthermore, the quotient map 
		$\mu: M\to \Delta$ is a tropical momentum map. 
			
	\item[2.] If we fix the equivalence class of the log affine polytope $\Delta$ and the $n$-tuple of Chern classes of the symplectic uncut $\widetilde M$ of $M$ (see Definition~\ref{defn: sympl uncut}), then the space of equivariant isomorphism classes of $(M,Z,\omega,\mu)$ is an affine space modeled on the vector space $H^2(\Delta,\log D)$.
	\end{enumerate}
	If, in addition, the action is Hamiltonian in the sense of Definition~\ref{torlog}, so that $(\Delta,D,\xi)$ has trivial affine monodromy, then we also have:
\begin{enumerate}
	\item[3.] 
	$\Delta$ is equivalent to a convex log affine polytope in a tropical welded space as in Section~\ref{weld}.   
\end{enumerate}
In particular, $(M,Z,\omega)$ is equivariantly isomorphic to 
	a toric Hamiltonian log symplectic manifold of the type 
	constructed in Corollary~\ref{sympcutcon}.

  \end{theorem}
\end{pp}

\begin{proof}
\mbox{}

\vspace{1ex}
\noindent {\it Part 1:}
The notion of symplectic uncut was introduced in~\cite{Meinrenken1} to give a simpler proof of the Delzant correspondence for toric symplectic manifolds with proper momentum maps. It was studied carefully in \cite{Karshon1} in order to characterize non-compact toric symplectic manifolds.  We modify this construction in Appendix~\ref{symp uncut} in order to apply it to log symplectic manifolds.

The symplectic uncut $(\widetilde{M}, \widetilde{Z}, \widetilde{\omega})$, as described in Definition~\ref{defn: sympl uncut}, is a modification of $(M,Z,\omega)$ along its submaximal orbit strata which renders the 
$T^n$-action into a principal action. It follows from Lemma~\ref{lemma: unfolded blow-up} that $\Delta$ is a compact manifold with corners and that $D=Z/T^n$ is a normal crossing divisor. 

Since the infinitesimal action $\rho$ is Lagrangian, it follows that the logarithmic 1-form  
$$
\iota_{\rho}\widetilde{\omega}\in \Omega^1(\widetilde M, \log \widetilde Z)\otimes\Tt^*
$$
is basic, defining a closed logarithmic 1-form $\xi \in \Omega^1(\Delta, \log D) \otimes \Tt^*$ which endows $\Delta$ with the structure of a log affine manifold.

To see that $\Delta$ is a log affine polytope, note that each component 
of $\del \Delta$, not contained in $D$, is the $T^n$-quotient of an exceptional boundary component of $\widetilde{M}$. 
By the normal form given in Lemma~\ref{lemma: normal form}, and by the computation in Example \ref{example: local}, we see that each exceptional boundary component is the zero locus of an affine linear function with linear part given by a primitive generator of the integral lattice in $\Tt$.
Finally, Proposition \ref{prop: normal form} implies that $\Delta$ satisfies the Delzant condition required by Definition~\ref{delzantcondition}.

\vspace{1ex}
\noindent{\it Part 2:}
The uncut construction is such that $(M,Z,\omega)$ may be obtained from $(\widetilde{M}, \widetilde{Z}, \widetilde{\omega})$ by symplectic cutting along the nonsingular faces of $\Delta$.
By Theorem~\ref{nontrivprinc}, fixing the equivalence class of $\Delta$ and the $n$-tuple of Chern classes of $\widetilde M$, the  equivariant isomorphism class of the uncut is determined up to translation of the class $[\widetilde{\omega}]\in H^2(\widetilde M,\log\widetilde{Z})$ by basic classes in $H^2(\Delta,\log D)$. Since equivariant isomorphisms descend under symplectic cutting, we obtain the result.

\vspace{1ex}
\noindent{\it Part 3:}
We first construct a tropical welded space $(X, D_X, \xi_X)$ which contains the polytope $(\Delta, D, \xi)$. As in Paragraph~\ref{pp:dualcubcx}, we associate a cubic complex $\sf{C}_{\Delta, D}$ to $(\Delta, D)$: its closed vertices correspond to the (affine) components of $\Delta \setminus D$, its edges correspond to the codimension 1 affine strata of $D$, and so on. By decorating the edges of $\sf{C}_{\Delta, D}$ with the residue vectors of $(\Delta, D, \xi)$ as in Paragraph~\ref{pp:admcubcx}, we obtain an admissible decorated cubic complex as in Definition~\ref{def:admcubcx}. By 
Proposition~\ref{prop: grandweld}, we construct the tropical welded space $(X, D_X, \xi_X)$ from the admissible decorated cubic complex.

We now construct an embedding $\iota$ of $\Delta$ into $X$ as a log affine submanifold; the embedding is unique up to an overall translation. 
Choose a basepoint $x_0$ in a component $\Delta_0$ of $\Delta \setminus D$, and choose $y_0$ in the interior of the corresponding component of $X \setminus D_X$.

We define $\iota(x_0)=y_0$, which fixes the translation ambiguity. For any other point $x \in \Delta \setminus D$, we choose a path $\gamma$ from $x_0$ to $x$ which, for convenience, is transverse to each component of $D$.   

We cover $\gamma$ by open sets in each of which we have 
coordinates such that $\xi$ is in normal form~\eqref{xinormform}. 
Applying Theorem~\ref{umplogaf}, we obtain a unique extension of 
$\iota$ to the union $U_\gamma$ of this cover of $\gamma$,
\begin{equation} \label{eq: local embed}
	\iota_\gamma: U_\gamma \to X,
\end{equation}
which is a local diffeomorphism onto its image and preserves the
log affine structure. We write $y = \iota_\gamma(x) \in X$.

We now show that the construction above is path-independent.  
Let $\gamma'$ be another such path from $x_0$ to $x$, yielding
a possibly different point $y' = \iota_{\gamma'}(x) \in X$. Then $\iota_{\gamma'}(\gamma') \cdot \iota_\gamma(\gamma^{-1})$ is a path from $y$ to $y'$.
Because $\Delta$ has trivial affine monodromy, the loop $\gamma' \cdot \gamma^{-1}$ has zero regularized length, implying 
\begin{equation} \label{eq: int}
	\int_{\iota_{\gamma'}(\gamma') \cdot \iota_\gamma(\gamma^{-1})} \xi_X = \int_{\gamma' \cdot \gamma^{-1}} \xi = 0.
\end{equation}
However, $y$ and $y'$ lie in the same affine stratum $A$ of $X$ (by the construction of $X$), so we can choose a path $\delta\subset A$ from $y$ to $y'$. Using the affine
structure on $A$, we have hence
$$
	y'-y = \int_{\delta} \xi_A = \int_{\iota_{\gamma'}(\gamma') \cdot \iota_\gamma(\gamma^{-1})} \xi_X = 0,
$$
where $\xi_A$ is the 1-form defining the affine structure on $A$,
obtained from $\xi$ by reduction and pull back as in Paragraph~\ref{pulbk}.
In this way, we obtain a local diffeomorphism of log affine manifolds onto its image
\begin{equation} \label{eq: embed}
	\iota: \Delta \to X.
\end{equation}

Because $\Delta$ is compact, \eqref{eq: embed} is a covering map, so it suffices to show that it is a 1-sheeted cover. Take $y_0 \in \iota(\Delta) \cap (X \setminus D_X)$ and $x_0, x'_0 \in \iota^{-1}(y_0)$. Because $x_0$ and $x'_0$ are in the same component of $\Delta \setminus D$, we may choose a path $\lambda$ from $x_0$ to $x'_0$ contained in that component. Since $\Delta$ has trivial monodromy, it follows that
$$
	x'_0 - x_0 = \int_\lambda \xi = 0,
$$
as required.  
\end{proof}

\appendix 
\section{Appendix}
\subsection{Local normal forms for log affine manifolds}\label{normaff}

\begin{pp}
  Let $(X,D,\xi)$ be a log affine manifold with $D$ of normal crossing
  type.  If $x\in X$ is such that precisely $k$ components $D_1,\ldots,
  D_k$ of $D$ meet $x$, then we may
  choose coordinates $(y_1,\ldots, y_n)$ in a neighborhood $U$ centered at $x$ such that $D$
  is the vanishing locus of the monomial $y_1\cdots y_k$.  The
  coordinates also provide a natural trivialization of the algebroid
  $TX(-\log D)$, i.e., the sections of $TX(-\log D)$ are generated by the vector fields
  $$
    y_1\del_{y_1}, \ldots, y_k\del_{y_k}, \del_{y_{k+1}},\ldots, \del_{y_{n}}.
  $$

\begin{prop}\label{coordsaff}
  There exist coordinates $(x_1,\ldots, x_n)$ such that $D$ is the vanishing locus of $x_1\cdots x_k$ 
  and 
 
  \begin{equation}
    \label{eq:33}
    \xi = \sum_{i=1}^k  x_i^{-1}dx_i\otimes v_i + \sum_{i=k+1}^n dx_{i}\otimes  v_{i},
  \end{equation}
  for $\{v_1,\ldots v_n\}$ a constant basis of $\RR^n$. 
\end{prop}
\begin{proof}
Begin with coordinates $(y_1,\ldots, y_n)$ in the neighborhood $U$ as above,  so that $D$ is given by $y_1\cdots y_k = 0$ and  
\[
\xi = \sum_{i=1}^k y_i^{-1} dy_i\otimes \alpha_i  + \sum_{i=k+1}^n dy_i \otimes\alpha_i,
\]
where $\alpha_i: U\to \RR^n$ are smooth vector-valued functions.   It follows from $d\xi=0$ that 
for each $i\leq k$, $\alpha_i$ is constant along the hypersurface $y_i=0$.  Denoting the basis $\{\alpha_1(0),\ldots, \alpha_n(0)\}$ by 
$\{v_1,\ldots, v_n\}$, we conclude
that 
\[
\xi - \left(\sum_{i=1}^k y_i^{-1} dy_i\otimes v_i  + \sum_{i=k+1}^n dy_i\otimes v_i \right) = dF,
\]
for $F:U\to \RR^n$ smooth and $F(0) = dF(0) = 0$. 

We now expand $F$ in terms of the basis $\{v_1,\ldots, v_n\}$,
\[
F = \sum_{i=1}^n f_i v_i,
\] 
and we define new coordinates on a possibly smaller neighborhood given by
\begin{equation}
  \label{eq:40}
  x_i = \begin{cases}

y_i e^{f_i} & \text{for} \quad i\leq k,\\
y_i + f_i   & \text{for} \quad i > k.
\end{cases}
\end{equation}
In these coordinates, we have the required expression~\eqref{eq:33}.
\end{proof}

\end{pp}

\begin{pp}
  In the coordinates provided by Proposition~\ref{coordsaff}, the 
  $\RR^n$-action on $X$ may be written as
  \begin{equation}
    \label{eq:34}
    u\cdot (x_1,\ldots x_n) = (e^{v_1^*(u)}x_1,\ldots e^{v_k^*(u)}x_k, x_{k+1} + v^*_{k+1}(u), \ldots, x_n + v^*_n(u)),
  \end{equation}
  where $\{v_1^*,\ldots v_n^*\}$ is the dual basis to $\{v_1,\ldots, v_n\}$.
\end{pp}

\begin{pp}

  We see from~\eqref{eq:33} that for each $i\leq k$, $v_i$
  is the residue of $\xi$ along $D_i$, and so is an invariant of the
  log affine manifold.  We also see that these residues are linearly
  independent, since they are contained in the basis $\{v_1,\ldots, v_n\}$.
  The remaining basis elements $\{v_i \mid i>k\}$, however, are not
  invariant and may be adjusted by a change of coordinates, as
  follows.
  
  If $\{v_1,\ldots, v_k, w_{k+1},\ldots, w_n\}$ is any other basis,
  define $\{r_{ij}, s_{ij}\}$ via
  \begin{equation}
  v_j = \sum_{i>k} r_{ij} w_i + \sum_{i\leq k} s_{ij} v_i,  \text{ for all } j>k.
  \end{equation}
  We then make the coordinate change 
  \begin{equation}
    \label{eq:36}
    \tilde x_i =    \begin{cases}
      x_i \exp({\sum_{j>k} s_{ij}  x_j}) & \text{ for }  i \leq k,\\
      \sum_{j>k} r_{ij} x_j & \text{ for } i>k. 
    \end{cases}
  \end{equation}
  As a consequence, we obtain 
  \begin{equation}
    \sum_{i=1}^k  x_i^{-1}dx_i\otimes v_i + \sum_{i=k+1}^n dx_{i}\otimes  v_{i}=
    \sum_{i=1}^k  \tilde x_i^{-1}d\tilde x_i\otimes v_i + \sum_{i=k+1}^n d\tilde x_{i}\otimes  w_{i}.
  \end{equation}

\end{pp}

\begin{prop}\label{normlog}
  Let $(X,D,\xi)$ be a log affine $n$-manifold and suppose that $x\in X$
  lies in
  exactly $k$ components $D_1,\ldots, D_k$ of the normal crossing
  divisor $D$. Then the residues $v_i\in \RR^n$ of $\xi$ along $D_i$ are
  linearly independent. In addition, for any extension $\{v_1,\ldots, v_k, v_{k+1},\ldots v_n\}$ 
  to a basis, there are coordinates near $x$ such that 
\begin{equation}\label{xinormform}
  \xi = \sum_{i=1}^k  x_i^{-1}dx_i\otimes v_i + \sum_{i=k+1}^n dx_{i}\otimes  v_{i}.
\end{equation}
\end{prop}

\begin{theorem}\label{umplogaf}
  Let $(X,D,\xi)$ and $p\in U\subset X$ be as above.  If we choose a
  basepoint $q\in U\setminus D$ and map it to any point $\phi(q)\in
  \RR^n$, then there is a unique extension of $\phi$ to a map   \begin{equation}
    \phi: U\to X_A
  \end{equation}
with values in the
log affine manifold $X_A$ constructed in~\eqref{eq: X_A},
where $A=\{v_1,\ldots, v_k\}$ is the set of residues of $\xi$ in
  $U$. This map is an isomorphism of log affine manifolds onto its
  image.
\end{theorem}
\begin{proof}
  Choose local coordinates $(x_1,\ldots, x_n)$ in $U$ so that $\xi$ is
  given by~\eqref{eq:33} and choose similar adapted coordinates
  $(y_1,\ldots, y_n)$ for $X_A$, so that its log 1-form $\xi_A$ has the form 
  \begin{equation}
    \label{eq:35}
    \xi_A = \sum_{i=1}^k  y_i^{-1}dy_i\otimes v_i + \sum_{i=k+1}^n dy_{i}\otimes  w_{i},
  \end{equation}
  for a basis $\{v_1,\ldots, v_k, w_{k+1},\ldots w_n\}$ of $\RR^n$.
  Using Proposition~\ref{normlog}, we change coordinates so that 
  \begin{equation}
    \label{eq:37}
    \xi_A =  \sum_{i=1}^k  \tilde y_i^{-1}d\tilde y_i\otimes v_i + \sum_{i=k+1}^n d\tilde y_{i}\otimes  v_{i}.
  \end{equation}
Therefore, the map $\phi$, which must satisfy $\phi^* \xi_A = \xi$, is given by 
  \begin{equation}
    \label{eq:38}
    \tilde y_i = \begin{cases}
      e^{c_i} x_i & \text{ for } i\leq k,\\
      x_i + c_i & \text{ for }  i>k,
    \end{cases}
  \end{equation}
  where the constants $\{c_i\in \RR\}$ are fixed by $\phi(p) = q$, as required.
\end{proof}

\subsection{Log symplectic reduction}

\begin{pp}
  Let $(M,Z,\omega)$ be a log symplectic manifold with an $S^1$-action
  generated by the vector field $\del_\theta$, and suppose it is
  Hamiltonian, in the sense that there is a function $\mu:M\to \RR$
  such that
  \begin{equation}\label{eq:44}
  i_{\del_\theta} \omega = - d\mu.
  \end{equation}
  We impose the following transversality condition: along $\widetilde M =
  \mu^{-1}(0)$, the composition $\widetilde{d\mu}$ of $d\mu$ with the anchor map
  $TM(-\log Z)\to TM$ 
  \begin{equation}
    \label{eq:43}
    \xymatrix{TM(-\log Z)|_{\widetilde M}\ar[r]^-{\widetilde{d\mu}}& \mu^*T_0\RR}
  \end{equation}
  must be surjective. In particular, this implies that $0$ is a
  regular value of $\mu$ and so $\widetilde M\subset M$ is a smooth
  hypersurface.  If $Z$ is a normal crossing divisor, \eqref{eq:43} is
  surjective if and only if $0$ is a regular value, not only for
  $\mu$, but also for the restriction of $\mu$ to any stratum of $Z$.
  As a result, the intersection $\widetilde Z = Z\cap \widetilde M$ is
  a normal crossing divisor in $\widetilde M$, and the kernel of the
  morphism~\eqref{eq:43} is the Lie algebroid $T\widetilde M(-\log
  \widetilde Z)$.

  If we assume further that $S^1$ acts freely and properly on $\widetilde
  M$, we obtain a normal crossing divisor $Z_0 = \widetilde Z/S^1$ in the
  quotient manifold $M_0=\widetilde M/S^1$. The quotient map $\pi:
  \widetilde M\to M_0$ induces a Lie algebroid morphism and an exact
  sequence over $\widetilde M$
  \begin{equation}
    \label{eq:444}
    \xymatrix{0\ar[r] & \IP{\del_\theta}\ar[r] & T\widetilde M(-\log \widetilde Z)\ar[r]^{\pi_*} & TM_0(-\log Z_0)\ar[r] & 0.}
  \end{equation}
  As in the usual Marsden-Weinstein symplectic reduction, the corank 1
  subalgebroid $T\widetilde M(-\log \widetilde Z)\subset TM(-\log Z)$
  is coisotropic with respect to the log symplectic form, whose
  restriction has rank 1 kernel generated by $\del_\theta$.
  Therefore, the pull back of $\omega$ to $\widetilde M$ is basic
  relative to $\pi$, and may be expressed as the pull back of a unique
  log symplectic form $\omega_0\in \Omega^2(M_0,\log Z_0)$, defining
  the logarithmic symplectic reduction.  We state a slight extension
  of this result for arbitrary free divisors.

\begin{prop}\label{logred}
  Let $(M,Z,\omega)$ be a log symplectic manifold, let $\mu$ be a
  Hamiltonian generating a circle action via~\eqref{eq:44}, and assume
  that~\eqref{eq:43} is surjective, which implies that $\mu^{-1}(0)$
  is smooth, with free divisor $\widetilde Z = \mu^{-1}(0) \cap Z$.
  Assume that the $S^1$-action on $\mu^{-1}(0)$ is free and proper, with
  quotient $M_0$ containing the free divisor $Z_0 = \widetilde Z/S^1$.
  Then the pull back of $\omega$ to $\mu^{-1}(0)$ is basic in the sense
  of~\eqref{eq:444}, and defines a logarithmic symplectic structure
  $(M_0,Z_0,\omega_0)$, called the \emph{log symplectic quotient}.
\end{prop}

\end{pp}

\subsection{Symplectic uncut} \label{symp uncut}

Let $(M,Z,\omega)$ be a compact orientable log symplectic manifold with corners and with $Z$ of normal crossing type, equipped with an effective $T^n$-action. We construct a log symplectic principal $T^n$-bundle 
$(\widetilde{M}, \widetilde{Z}, \widetilde{\omega})$ over the quotient $M / T^n$, called the symplectic uncut (Definition \ref{defn: sympl uncut}), by applying the compressed  blow-up operation to $M$ along its submaximal orbit type strata, following an idea of Meinrenken~\cite{Meinrenken1}.  This operation is inverse to the symplectic cut introduced in Section~\ref{scut}, which may be 
applied to $\widetilde{M}$ along $\widetilde{Z}$ to recover $M$.

\subsubsection{Compressed blow-up}

To define the compressed blow-up operation (Definition \ref{defn: unf blow-up}), we first recall the blow-up operation and the boundary compression construction.

\begin{pp}

For $M$ a manifold and $K$ a submanifold of codimension $2$, we denote the \emph{real oriented blow-up} of $M$ along $K$ by $\blowup{M}{K}$. We denote the exceptional divisor by $E$, which is a boundary component of $\blowup{M}{K}$. The blow-down map
\begin{equation}
	p: \blowup{M}{K} \to M
\end{equation}
restricts to a diffeomorphism on $\blowup{M}{K} \setminus E$ and expresses $E$ as an $S^1$-fiber bundle over $K$.

\end{pp}

\begin{pp}
We now define a \emph{compression} operation closely related to the unfolding defined in \cite{MR1748286}. Let $M$ be a smooth manifold with corners and let $Z$ be a smooth component of the boundary. Let $U \cong Z \times [0, \epsilon)$ be a tubular neighborhood of $Z$, and let $V = M \setminus Z$. Then $M$ is the fibered coproduct, or gluing, of $U$ and $V$ along $U \cap V$. Consider the open embedding
\begin{equation}
	\iota: U \cap V \cong Z \times (0, \epsilon) \to \widetilde{U} = Z \times \left[0, \tfrac{1}{2}\epsilon^2\right), \qquad (p, r) \mapsto \left( p, \tfrac{1}{2}r^2 \right).
\end{equation}
The fibered coproduct of $\widetilde{U}$ and $V$ along $U \cap V$ is a smooth manifold with corners $\widetilde{M}$. The identity map on $V$ extends to a smooth homeomorphism $q: M \to \widetilde{M}$ such that the restriction to $Z$ is a diffeomorphism onto its image $\widetilde{Z} = q(Z)$. The triple $(\widetilde{M}, \widetilde{Z}, q)$ is called the \emph{compression} of $(M, Z)$.

\end{pp}

We now compress the blow-up along its exceptional divisor to obtain the following. 

\begin{defn} \label{defn: unf blow-up}

Let $M$ be a manifold with corners and let $K \subset M$ be a submanifold of codimension $2$. Then the \textbf{compressed blow-up} of $M$ along $K$ is the compression $(\widetilde{M}_K, \widetilde{E}, q)$ of $(\blowup{M}{K}, E)$.

If $S \subset M$ is a submanifold that intersects $K$ cleanly, the \textbf{compressed proper transform} of $S$ is the submanifold
$$
	\widetilde{S} = \overline{(q\circ p^{-1})(S \setminus K)} \subset \widetilde{M}_K.
$$

\end{defn}

\begin{example} \label{example: local blow-up} Let $M=\RR^2$ and $K=\{0\}\subset M$. The real oriented blow-up is then $\blowup{M}{K}=[0,\infty)\times S^1$; using standard coordinates $(r,\theta)$ for $\blowup{M}{K}$ and the complex coordinate on $M$, the blow-down map is $p(r,\theta) = r e^{i\theta}$.   The compressed blow-up is also given by $\widetilde{M}_K = [0,\infty)\times S^1$, and receives a smooth homeomorphism from the blow-up $q(r,\theta) = (\tfrac{1}{2}r^2,\theta)$. That is, $q: \blowup{M}{K} \to \widetilde{M}_K$ is smooth with a continuous inverse.
\end{example}

\begin{pp} \label{codim 2 strata}
We apply now the compressed blow-up operation to the submaximal strata 
of an effective $T^n$-action.
Let $M$ be a $2n$-dimensional compact orientable manifold with corners equipped with an effective $T^n$-action.  By \cite[Corollary B.48]{MR1929136}, the fixed point set of a circle subgroup $H \subset T^n$ is a disjoint union of closed submanifolds of codimension at least $2$.

Let $H_1, H_2, \ldots, H_m \subset T^n$ be the circle subgroups such that the fixed point set $M^{H_i}$ contains at least one codimension $2$ component (there are finitely many $H_i$, because $M$ is compact, and therefore has finitely many orbit type strata). Let $K_i$ be the union of the codimension $2$ components of $M^{H_i}$.

Let $\widetilde{M}_{K_1}$ be the compressed blow-up of $M$ along $K_1$. The compressed proper transform
$$
\widetilde{K_2} \subset \widetilde{M}_{K_1}
$$
is a codimension $2$ submanifold of $\widetilde{M}_{K_1}$, so we may apply the compressed blow-up operation to $\widetilde{M}_{K_1}$ along $\widetilde{K_2}$. Continuing in this way, we obtain the following.

\begin{defn} \label{it unf blow-up}
Let $M$ be a $2n$-dimensional compact orientable manifold with corners with an effective $T^n$-action.  The \textbf{iterated compressed blow-up} $\widetilde{M}$ is the manifold with corners obtained from $M$ by successively applying the compressed blow-up operation to the fixed point sets $\{K_1, K_2, \ldots, K_m\}$.
\end{defn}

The iterated compressed blow-up $\widetilde{M}$ receives the compression map
$$
	q: \blowup{M}{K} \to \widetilde{M},
$$
where $\blowup{M}{K}$ is the iterated real oriented blow-up of $M$ along $\{K_1, K_2, \ldots, K_m\}$.

\end{pp}

\begin{lemma} \label{lemma: unfolded blow-up}
	Let $M$ be a $2n$-dimensional compact orientable manifold with corners endowed with an effective $T^n$-action. The iterated compressed blow-up $\widetilde{M}$ in Definition \ref{it unf blow-up} is a principal 
$T^n$-bundle over the quotient $\Delta = M / T^n$, and therefore $\Delta$ is a compact orientable manifold with corners. Furthermore, the following diagram commutes
\begin{equation} \label{eq: 56}
	\begin{aligned}
	\xymatrix{
		\blowup{M}{K} \ar[d]_-{q} \ar[r]^-{p} & M \ar[d] \\
		\widetilde{M} \ar[r] & \Delta
	}
	\end{aligned}
\end{equation}
and the blow-down and folding maps $p, q$  are $T^n$-equivariant. 

In addition, if $Z$ is a $T^n$-invariant normal crossing divisor in $M$, then the compressed proper transform $\widetilde{Z}$ is a normal crossing divisor in $\widetilde{M}$, and $D = Z / T^n$ is a normal crossing divisor in $\Delta$.
\end{lemma}

\begin{proof}
	Because $M$ is orientable, $T^n$ acts freely on the principal orbit stratum, and every point with non-trivial stabilizer is contained in some $K_i$.

	For any point $x \in K_{i_1} \cap K_{i_2} \cap \ldots \cap K_{i_k}$, there exists a $T^n$-invariant neighborhood $U_x$ that is equivariantly isomorphic to a neighborhood of
$$
	(x, 0)\in (U_x \cap K_{i_1} \cap K_{i_2} \cap \ldots \cap K_{i_k}) \times \CC^k,
$$
where $\CC^k$ is endowed with the standard $T^k$-action. It follows that
$q ( p^{-1}(x))$ is equivariantly isomorphic to a $k$-torus $T^k$, and $q ( p^{-1}(U_x))$ is equivariantly isomorphic to a neighborhood of
$$
	q ( p^{-1}(x)) \cong \{(x, 0)\} \times T^k \in (U_x \cap K_{i_1} \cap K_{i_2} \cap \ldots \cap K_{i_k}) \times  \RR_+^k \times T^k,
$$
where $\RR_+ = [0, \infty)$. Therefore, the induced $T^n$-action on 
$\widetilde{M}$ is free and proper, and diagram~\eqref{eq: 56} commutes.

Since $M$ is orientable, it follows that the iterated compressed blow-up $\widetilde{M}$ is orientable, and therefore $\Delta$ is orientable.

Let $Z_1, Z_2, \ldots, Z_l$ be components of $Z$. The intersection
$$
	L = \bigcap_{i =1, \ldots, l} Z_i
$$
is $T^n$-invariant. By induction on $l$, the $T^n$-action on $L$ is effective.

The isotropy action of $H_i$ on $T_xM$ descends to the normal space $N_xK_i$. Up to a linear transformation, $H_i$ acts on $N_xK_i \cong \CC$ by the standard $S^1$-action. Because the $T^n$-action on $L$ is effective, it follows that the $T^n$-invariant subspace $T_xL \subset T_xM$ is transverse to $T_xK_i$. Therefore $\widetilde{Z}$ is a normal crossing divisor in $\widetilde{M}$ and $D$ is a normal crossing divisor in $\Delta$. 
\end{proof}

\subsubsection{Local normal form}

Given a log symplectic manifold $(M, Z, \omega)$ equipped with an effective $T^n$-action, we show that intersections among the codimension 2 circle fixed point sets $\{K_1, K_2, \ldots, K_m\}$ are log symplectic submanifolds. We also show the existence of local normal forms for neighborhoods of points on such submanifolds. We begin with a convenient definition of log symplectic submanifold.

\begin{defn}
Let $(M,Z,\omega)$ be a log symplectic manifold. An embedded submanifold $i:N \hookrightarrow M$ is called a \textbf{log symplectic submanifold} if
$Ti:TN\to TM$ is transverse to the anchor map
$TM(-\log Z) \to TM$ and $\omega$ is non-degenerate 
on the resulting intersection.
\end{defn}

If $Z$ is of normal crossing type, then the 
above transversality condition
is equivalent to the condition that $N$ be transverse
to all possible intersections among components of $Z$.  This means that $Z_N = N\cap Z$ is a normal crossing divisor in $N$, and then $N$ is log symplectic when $\omega$ is non-degenerate when pulled back to the subbundle 
\[
TN(-\log Z_N)\hookrightarrow TM(-\log Z).
\]

\begin{lemma} \label{lemma: normal form}

Let $(M,Z,\omega)$ be a $2n$-dimensional compact orientable log symplectic manifold with $Z$ of normal crossing type, equipped with an effective 
$T^n$-action. Let $K$ be a codimension $2$ component of the fixed point set $M^H$ of a circle subgroup $H \subset T^n$. 

Then $K$ is a log symplectic submanifold. For every point $x \in K$, there exists a $T^n$-invariant neighborhood $U_x$ such that $U_x$ is isomorphic to a neighborhood of
\begin{equation}
	(x, 0) \in M^H \times \CC,
\end{equation}
where $\CC$ is equipped with the symplectic structure
$$
	\frac{i}{2}dz\wedge d\overline{z}
$$
and $H$ acts on $\CC$ by the standard $S^1$-action.
\end{lemma}

\begin{proof}
Let $Z_1, Z_2, \ldots, Z_\ell$ be components of $Z$. As shown in the proof of Lemma~\ref{lemma: unfolded blow-up}, $K$ is transverse to the intersection
$$
	\bigcap_{i =1, \ldots, \ell} Z_i.
$$
This implies that $K \cap Z$ is a normal crossing divisor in $K$. For every point $x\in K$, the log tangent space
$$
	T_x(M, - \log Z)
$$
is a symplectic vector space, and the $H$-invariant subspace
$$
	T_x(K, - \log (K \cap Z) )
$$
is a symplectic subspace. This shows that $K$ is a log symplectic submanifold.

Taking a smaller neighborhood, if necessary, we may assume that $U_x$ is contractible and
$$
	U_x \cong (U_x \cap K) \times D,
$$
where $D \subset \CC$ is an open disk. We use the Moser method to show that $(U_x, \omega)$ is equivariantly symplectomorphic to $(U_x \cap K, \omega_K) \times (D, \sigma)$, where $\omega_K$ is the induced log symplectic structure on $U_x \cap K$ and $D \subset \CC$ is equipped with the symplectic structure $\sigma = \frac{i}{2}dz\wedge d\overline{z}$ and the standard $S^1$-action. We write
$$
	\omega_t \coloneqq (1-t)\omega + t (\omega_K \oplus \sigma), \qquad 0\leq t \leq 1.
$$
Because $K \pitchfork (Z_i \cap Z_j)$, for small enough $U_x$, we have
$$
	\bigoplus_{i, j} H^0(U_x \cap Z_i \cap Z_j) = \bigoplus_{i, j} H^0(K \cap U_x \cap Z_i \cap Z_j),
$$
and therefore
$$
	H^2(U_x, \log Z) = H^2(K \cap U_x, \log (K \cap Z)).
$$
This implies $[\omega_0] = [\omega_t] \in H^2(U_x, \log (U_x \cap Z))$, and so
$$
	\omega_0 - \omega_1 = t d \tau'
$$
for some $\tau' \in \Omega^1(U_x, \log (U_x \cap Z))$.

Choose a chart centered at $x$ such that
\begin{equation} \label{eq: 42}
	\tau' |_x = \sum_{i = 1}^\ell c_i \frac{d x_i}{x_i} + \sum_{i = \ell+1}^{2n} c_i d x_i.
\end{equation}
We use the right hand side of \eqref{eq: 42} to define a closed logarithmic 1-form $\tau_0$ on $U_x$. Consequently,
$$
	\omega_0 - \omega_1 = t d \tau' = t d (\tau' - \tau_0) = t d \tau,
$$
where $\tau \coloneqq \tau' - \tau_0$ has the property that $\tau |_x = 0$. We average $\tau$ over $T^n$, getting a new map, also called $\tau$,
which is $T^n$-equivariant. The usual Moser method yields the desired result.
\end{proof}

By Lemma~\ref{lemma: normal form}, we deduce that $K$ is a compact orientable log symplectic submanifold of dimension $2n-2$ with an induced effective action by the torus $T^n / H$, so we obtain the following result by induction on dimension.

\begin{prop} \label{prop: normal form}
Let $(M,Z,\omega)$ be a $2n$-dimensional compact orientable log symplectic manifold with $Z$ of normal crossing type, equipped with an effective 
$T^n$-action. Let $\{K_1,\ldots, K_m\}$ be the submaximal orbit type strata as described in \ref{codim 2 strata}. Then for $1 \leq i_1, i_2, \ldots, i_k \leq m$, where $i_j$'s are distinct,
$$
	N = K_{i_1} \cap K_{i_2} \cap \ldots \cap K_{i_k}
$$
is a log symplectic submanifold with the induced action by the torus
$$
	T^n ~/~ H_{i_1} H_{i_2} \cdots H_{i_k}.
$$
For a point $x \in N$, there exists a $T^n$-invariant neighborhood $U_x$ that is equivariantly isomorphic to a neighborhood of
\begin{equation}
	(x, 0) \in (N \cap U_p) \times \CC^k,
\end{equation}
where $\CC^k$ is equipped with the symplectic structure
$$
	\frac{i}{2} \left( dz_1\wedge d\overline{z}_1 + dz_2\wedge d\overline{z}_2 + \ldots + dz_k\wedge d\overline{z}_k \right)
$$
and each $H_{i_p}$ acts only on the $p^{th}$ factor of $\CC^k$ by the standard $S^1$-action.
\end{prop}

\begin{example} \label{example: local}
	We continue with Example~\ref{example: local blow-up}.  Equip $M=\RR^2$ with the standard symplectic form ($\sigma=\tfrac{i}{2}dz\wedge d\bar z$, in 
complex coordinates)  which is invariant by rotations of the plane. The blow-up at the origin then has induced presymplectic structure $rdr\wedge d\theta$, while the compressed blow-up is endowed with the symplectic structure $\widetilde{\sigma} = {d} t \wedge {d} \theta$ in standard coordinates $(t,\theta)$ for $\widetilde{M}_K = [0,1)\times S^1$. The Hamiltonian function $t$ generates the induced $S^1$-action on 
$\widetilde{M}_K$. Finally, the symplectic cut of $(\widetilde{M}_K, \widetilde\sigma )$ with respect to the function $t$ is equivariantly isomorphic to $(M, \sigma)$.
\end{example}

\subsubsection{Symplectic uncut}

We use Proposition \ref{prop: normal form} to show that there is an induced $T^n$-invariant log symplectic structure on the compressed blow-up $\widetilde{M}$. By Lemma \ref{lemma: unfolded blow-up}, $\widetilde{M}$ defines a log symplectic principal $T^n$-bundle, which we call the 
symplectic uncut (see Definition \ref{defn: sympl uncut} below).

\begin{prop} \label{prop: sympl uncut}
	Let $(M,Z,\omega)$ be a $2n$-dimensional compact orientable log symplectic manifold with corners, where $Z$ is of normal crossing type, equipped with an effective $T^n$-action. Let $\Delta = M / T^n$ and 
	$D = Z / T^n$.

Then the iterated compressed blow-up $(\widetilde{M}, \widetilde{Z})$ in Lemma \ref{lemma: unfolded blow-up}, as a principal $T^n$-bundle over 
$(\Delta, D)$, is endowed with a unique $T^n$-invariant log symplectic structure $\widetilde{\omega}$ such that the pull back of $\widetilde{\omega}$ to $\blowup{M}{K}$ coincides with the pull back of $\omega$ to $\blowup{M}{K}$, i.e., we have $p^*\omega = q^*\widetilde{\omega}$ in diagram~\eqref{eq: 56}.
\end{prop}

\begin{proof}
For a point
$$
	x \in N = K_{i_1} \cap K_{i_2} \cap \ldots \cap K_{i_k} \subset M,
$$
let $U_x \subset M$ be a $T^n$-invariant neighborhood that is isomorphic to a neighborhood of
$$
	(x, 0) \in (N \cap U_x) \times \CC^k.
$$
as in Proposition \ref{prop: normal form}. In the proof of Lemma \ref{lemma: unfolded blow-up}, $q (p^{-1}(U_x)) \subset \widetilde{M}$ is equivariantly isomorphic to a neighborhood of 
$$
	(x, 0) \times T^k \subset (N \cap U_x ) \times \left(\RR_+ \times S^1 \right)^k,
$$
where $\RR_+ = [0, \infty)$. In these coordinates, the $T^n$-invariant symplectic form
$$
	\widetilde{\omega} = \omega_N \oplus \left( d t_1 \wedge d\theta_1 + d t_2 \wedge d\theta_2 + \ldots + d t_k\wedge d\theta_k \right)
$$
satisfies $p^*\omega = q^*\widetilde{\omega}$.
\end{proof}

\begin{defn} \label{defn: sympl uncut}
Let $(M,Z,\omega)$ be a compact orientable log symplectic manifold with corners with $Z$ of normal crossing type, equipped with an effective 
$T^n$-action. The triple $(\widetilde{M}, \widetilde{Z}, \widetilde{\omega})$, as a log symplectic principal $T^n$-bundle over $(\Delta, D)$, is called the \textbf{symplectic uncut} of $(M, Z, \omega)$.
\end{defn}

\subsection{Logarithmic de Rham cohomology}\label{Mazzeo-Melrose}

\begin{pp}
Let $M$ be a smooth manifold and $Z$ a codimension $1$ closed hypersurface
of $M$.
Let $TM(-\log Z)$ be the log tangent algebroid. Then
  the residue exact sequence of de Rham complexes
\[
\xymatrix{0\ar[r]& \Omega^k (M)\ar[r]& \Omega^k(M,\log Z)\ar[r]^-{\Res} & \Omega^{k-1} (Z) \ar[r] & 0}
\]
smoothly splits as a sequence of complexes, giving a natural decomposition of cohomology groups  
\begin{equation}
  \label{eq:42}
  H^k(M,\log Z)\cong  H^k(M)\oplus H^{k-1}(Z),
\end{equation}
recovering the well-known isomorphism of Mazzeo-Melrose~\cite{MR1734130}.
\end{pp}

\begin{pp}
  If $Z$ is the degeneracy locus of a log symplectic structure, then
  $Z$ has a foliation $F$ by codimension 1 symplectic leaves, and the
  Poisson Lie algebroid is isomorphic to $TM(-\log Z,F)$, the sheaf of
  vector fields tangent to $Z$ and to $F$.  A local computation shows
  that the morphism of complexes dualizing the morphism of algebroids
  $TM(-\log Z, F)\to TM(-\log Z)$, namely,
  \[
  \Omega^\bullet(M, \log Z)\to \Omega^{\bullet}(M, \log Z,F),
  \]
  is a quasi-isomorphism\footnote{A quasi-isomorphism of chain complexes is a morphism between the complexes that induces an isomorphism of their homologies.} and so the Poisson cohomology coincides with
  the logarithmic cohomology of the hypersurface $Z$.
\end{pp}

\begin{pp}
  If there are several smooth hypersurfaces $Z_1,\ldots,
  Z_k$ which intersect in a normal crossing fashion, for example the
  boundary of a manifold with corners, the Mazzeo-Melrose theorem
  easily generalizes to give (for the free divisor $Z = Z_1+\ldots + Z_k$)
  \begin{equation}
  H^k(M,\log Z) \cong H^k(M) \oplus \sum_i H^{k-1}(Z_i) \oplus \sum_{i<j}
  H^{k-2}(Z_i\cap Z_j)\oplus \cdots.
  \end{equation}
Furthermore, if $Z$ is the degeneracy locus of a log symplectic structure $\omega \in \Omega^2(M, \log Z)$, 
  then the corresponding Poisson structure $\pi = \omega^{-1} \in C^\infty(M, \wedge^2 TM(-\log Z))$ defines a morphism of Lie algebroids
$$
	\pi: T^* M \to TM(-\log Z), \qquad \alpha \mapsto i_\omega \pi,
$$
which dualizes to a quasi-isomorphism of cochain complexes:
$$
	(\Omega^\bullet(M, \log Z), d)\to (\mathfrak{X}^{\bullet}(M), [\pi, \cdot]).
$$
It follows that the Poisson cohomology is isomorphic to the logarithmic de Rham cohomology of the hypersurface arrangement. This allows a computation of the Poisson cohomology by way of the Mazzeo-Melrose isomorphism.
\end{pp}

\subsection{Poisson toric variety and the tropicalization map} \label{sec: CaineToric}

We outline here the relationship between the tropical momentum map introduced in this paper and the notion of extended tropicalization which exists in tropical geometry. To see the connection, we use a canonical real Poisson structure on any smooth complex toric variety introduced by Caine \cite{MR2859234}.

\begin{pp}

First, we recall the construction of Caine's Poisson toric variety. Let 
$\mathfrak{t}^n$ be the Lie algebra of the real $n$-torus $T^n$.
Let $\Sigma$ be a complete simplicial fan with $d$ 1-dimensional cones. We assume that $\Sigma$ is dual to a Delzant polytope, so that the resulting toric variety is smooth. We label each 1-dimensional cone with its integral generating vector $v_i \in \mathfrak{t}^n$, 
$i=1, \ldots, d$. The linear projection
defined by 
$\mathfrak{t}^d\ni e_i \mapsto v_i \in \mathfrak{t}^n$ for 
$i = 1, \ldots, d$, where $\{e_1, \ldots, e_d\}$ is the standard
basis in $\mathfrak{t}^d \cong \mathbb{R}^d$, induces the short exact sequence of Lie algebras
$$
	0 \longrightarrow \mathfrak{n} \longrightarrow \mathfrak{t}^d \longrightarrow \mathfrak{t}^n \longrightarrow 0.
$$

Now let $T_\CC^d$ be the complex torus with the standard action on $\CC^d$, and let $N$ be the $(d-n)$-dimensional torus integrating the Lie algebra $\mathfrak{n}$. The complex subtorus $N_\CC \subset T_\CC^d$ acts effectively on $\CC^d$, and there is an open dense orbit stratum $U \subset \CC^d$ such that $N_\CC$ acts freely. The compact holomorphic toric variety $\mathcal{X}$ of the simplicial fan $\Sigma$ is the quotient
$$
	\mathcal{X} = U / N_\CC,
$$
which is equipped with the residual $T_\CC^n$-action.

Now, if we consider the real Poisson structure
$$
	\Pi = \sum_{k=1}^d i \overline{z}_k \frac{\partial}{\partial \overline{z}_k} \wedge z_k \frac{\partial}{\partial z_k}
$$
on $\CC^d$, which is invariant under the $N_\CC$-action, then $\Pi$ descends to a Poisson structure $\pi$ on $\mathcal{X}$. The Poisson manifold $(\mathcal{X}, \pi)$ is what Caine called Poisson toric variety. Caine proved that $\pi$ is generically non-degenerate and has quadratic degeneracy along the submaximal orbits.

Because the simplicial fan $\Sigma$ is due to a Delzant polytope, it means that each $n$-dimensional cone of $\Sigma$ can be mapped to the standard positive orthant of $\mathfrak{t}^n$ by an automorphism of the integral lattice $\ZZ^n \subset \mathfrak{t}^n$. If  the standard positive orthant of $\mathfrak{t}^n$ is indeed a cone of $\Sigma$, then the Caine Poisson structure $\pi$ takes the form
$$
	\pi = \sum_{k=1}^n i \overline{Z}_k \wedge Z_k,
$$
where $\left\{Z_1, Z_2, \ldots, Z_n \right \}$ are the holomorphic vector fields on $\mathcal{X}$ that generate the action of $T^n_\CC \cong \left(\CC^* \right)^n$. 
\end{pp}

\begin{pp}
Our key observation is that if we carry out a real oriented blow-up along the submaximal orbits of $\mathcal{X}$, then the blow-up space is a smooth manifold with corners $\widetilde{\mathcal{X}}$, and $\pi$ lifts to a Poisson structure $\widetilde{\pi}$ on $\widetilde{\mathcal{X}}$ such that $\widetilde{\omega} = \widetilde{\pi}^{-1}$ is log symplectic and degenerates along the boundary $\partial \widetilde{\mathcal{X}}$. Furthermore, the induced $T^n$-action on $\widetilde{\mathcal{X}}$ is free, proper, and Lagrangian, so by Theorem~\ref{delzant}, the quotient $X = \widetilde{\mathcal{X}} / T^n$ is equipped with a log affine structure $\xi$ which degenerates along the boundary $\partial X$, and the quotient map
\begin{equation} \label{eq: trop}
	\mu: \left(\widetilde{\mathcal{X}}, \partial \widetilde{\mathcal{X}}, \widetilde{\omega}\right) \to (X, \partial X, \xi).
\end{equation}
is a tropical momentum map as in Definition~\ref{torlog}. By Proposition~\ref{prop: semi-local}, the log affine manifold $(X, \partial X, \xi)$ defines a labeled simplicial fan, which is simply the labeled simplicial fan $\Sigma$ of the toric variety $\mathcal{X}$.
\end{pp}

\begin{pp}
Now let $U_\mathcal{X} \cong T_\CC^n$ be the open dense orbit of $\mathcal{X}$. We have the natural tropicalization map $\nu: U_\mathcal{X} \to \mathfrak{t}^n$. Upon choosing a basis of $\mathfrak{t}^n$, the tropicalization map, see, e.g., Definition 1.1 of \cite{MR2102998}, takes the form
$$
	\nu: U_\mathcal{X} \cong (\CC^*)^n \to \mathfrak{t}^n \cong \RR^n, \qquad (z_1, \ldots, z_n) \mapsto (\log |z_1|, \ldots, \log |z_n|).
$$
The extended tropicalization map introduced by Kajiwara~\cite{MR2428356} and Payne~\cite{MR2511632} is the natural extension:
$$
	\nu: \mathcal{X} \to \mathbf{Trop}(\mathcal{X}),
$$
where the codomain $\mathbf{Trop}(\mathcal{X})$, called the 
\textit{tropicalization} of $\mathcal{X}$, is a natural compactification of $\mathfrak{t}^n$. From our perspective, if we regard the the simplicial fan $\Sigma$ of $\mathcal{X}$ as an admissible decorated cubic complex, then $\mathbf{Trop}(\mathcal{X})$ is nothing but the log affine manifold $X$, which is constructed from $\Sigma$ by the welding procedure in Proposition~\ref{prop: grandweld}.

We then have the following commutative diagram:
\begin{equation} \label{eq: trop triangle}
	\xymatrix{
		\widetilde{\mathcal{X}} \ar[dr]_-{\mu} \ar[r]^-{p}& \mathcal{X} \ar[d]^-{\nu} \\
		& X \cong \mathbf{Trop}(\mathcal{X}),
	}
\end{equation}
where $p: \widetilde{\mathcal{X}} \to \mathcal{X}$ is the blow-down map.

\begin{example}
The complex projective space $\CC P^1$ may be thought as the quotient of $\CC^2$ by the diagonal $\CC^*$-action. The residual anti-diagonal 
$\CC^*$-action renders $\CC P^1$ into a toric variety. If we use the homogeneous coordinates $z_1$ and $z_2$ on $\CC P^1$ with $z_1 = \frac{1}{z_2}$, then
$$
	\pi =  i \overline{z}_1 \frac{\partial}{\partial \overline{z}_1} \wedge z_1 \frac{\partial}{\partial z_1} = i \overline{z}_2 \frac{\partial}{\partial \overline{z}_2} \wedge z_2 \frac{\partial}{\partial z_2}.
$$
Let $\widetilde{\mathcal{X}}$ be the blow-up of $\CC P^1$ along the two points $z_1 = 0$ and $z_2 = 0$. We cover $\widetilde{\mathcal{X}}$ by the charts
$$
	\begin{aligned}
		(r_k = |z_k|, \theta_k = \arg z_k), \qquad r_k \geq 0, \quad \theta_k \in \RR / 2\pi \ZZ, \quad k = 1, 2,\\
	\end{aligned}
$$
the maps in the commutative diagram \eqref{eq: trop triangle} are as follows:
$$
	\begin{aligned}
		& p: \widetilde{\mathcal{X}} \to \CC P^1, \qquad (r_k, \theta_k) \mapsto z_k = r_k e^{i\theta_k}, \quad k = 1, 2; \\
		& \nu: \CC P^1 \to X, \qquad z_k \mapsto r_k = |z_k|, \quad k = 1, 2; \\
		& \mu: \widetilde{\mathcal{X}} \to X, \qquad (r_k, \theta_k) \mapsto r_k, \quad k = 1, 2.
	\end{aligned}
$$
\end{example}

\end{pp}

\bibliographystyle{hyperamsplain} 
\bibliography{TropicMoment} 

\providecommand{\bysame}{\leavevmode\hbox to3em{\hrulefill}\thinspace}
\providecommand{\MR}{\relax\ifhmode\unskip\space\fi MR }
\providecommand{\MRhref}[2]{%
  \href{http://www.ams.org/mathscinet-getitem?mr=#1}{#2}
}
\providecommand{\href}[2]{#2}
\begin{thebibliography}{10}

\bibitem{MR642416}
M.~F. Atiyah, \emph{Convexity and commuting {H}amiltonians},
  \href{http://dx.doi.org/10.1112/blms/14.1.1}{Bull. London Math. Soc.
  \textbf{14} (1982)}, no.~1, 1--15.

\bibitem{MR2859234}
A.~Caine, \emph{Toric {P}oisson structures}, Mosc. Math. J. \textbf{11} (2011),
  no.~2, 205--229, 406.

\bibitem{MR1748286}
A.~Cannas~da Silva, V.~Guillemin, and C.~Woodward, \emph{On the unfolding of
  folded symplectic structures},
  \href{http://dx.doi.org/10.4310/MRL.2000.v7.n1.a4}{Math. Res. Lett.
  \textbf{7} (2000)}, no.~1, 35--53.

\bibitem{MR998124}
T.~J. Courant, \emph{Dirac manifolds},
  \href{http://dx.doi.org/10.1090/S0002-9947-1990-0998124-1}{Trans. Amer. Math.
  Soc. \textbf{319} (1990)}, no.~2, 631--661.

\bibitem{MR984900}
T.~Delzant, \emph{Hamiltoniens p\'eriodiques et images convexes de
  l'application moment}, Bull. Soc. Math. France \textbf{116} (1988), no.~3,
  315--339.

\bibitem{MR2289207}
M.~Einsiedler, M.~Kapranov, and D.~Lind, \emph{Non-{A}rchimedean amoebas and
  tropical varieties}, \href{http://dx.doi.org/10.1515/CRELLE.2006.097}{J.
  Reine Angew. Math. \textbf{601} (2006)}, 139--157.

\bibitem{GLP1}
M.~Gualtieri, S.~Li, and B.~Pym, \emph{The Stokes groupoids},
  \href{http://arxiv.org/abs/arXiv:1305.7288}{{\tt arXiv:1305.7288}}.

\bibitem{MR3100779}
M.~Gualtieri and B.~Pym, \emph{Poisson modules and degeneracy loci},
  \href{http://dx.doi.org/10.1112/plms/pds090}{Proc. Lond. Math. Soc. (3)
  \textbf{107} (2013)}, no.~3, 627--654.

\bibitem{MR664117}
V.~Guillemin and S.~Sternberg, \emph{Convexity properties of the moment
  mapping}, \href{http://dx.doi.org/10.1007/BF01398933}{Invent. Math.
  \textbf{67} (1982)}, no.~3, 491--513.

\bibitem{MR1929136}
V.~Guillemin, V.~Ginzburg, and Y.~Karshon,
  \href{http://dx.doi.org/10.1090/surv/098}{\emph{Moment maps, cobordisms, and
  {H}amiltonian group actions}}, Mathematical Surveys and Monographs, vol.~98,
  American Mathematical Society, Providence, RI, 2002. Appendix J by Maxim
  Braverman.

\bibitem{MR3250302}
V.~Guillemin, E.~Miranda, and A.~R. Pires, \emph{Symplectic and {P}oisson
  geometry on {$b$}-manifolds},
  \href{http://dx.doi.org/10.1016/j.aim.2014.07.032}{Adv. Math. \textbf{264}
  (2014)}, 864--896.

\bibitem{Guillemin3}
V.~Guillemin, E.~Miranda, A.~R. Pires, and G.~Scott, \emph{Toric actions on
  {b}--symplectic manifolds},
  \href{http://arxiv.org/abs/arXiv:1309.1897v3}{{\tt arXiv:1309.1897v3}}.

\bibitem{MR0375320}
W.~Heil, \emph{Elementary surgery on {S}eifert fiber spaces}, Yokohama Math. J.
  \textbf{22} (1974), 135--139.

\bibitem{MR1344739}
G.~Hetyei, \emph{On the {S}tanley ring of a cubical complex},
  \href{http://dx.doi.org/10.1007/BF02570709}{Discrete Comput. Geom.
  \textbf{14} (1995)}, no.~3, 305--330.

\bibitem{MR2428356}
T.~Kajiwara, \href{http://dx.doi.org/10.1090/conm/460/09018}{\emph{Tropical
  toric geometry}}, Toric topology, Contemp. Math., vol. 460, Amer. Math. Soc.,
  Providence, RI, 2008, pp.~197--207.

\bibitem{Karshon1}
Y.~Karshon and E.~Lerman, \emph{Non-compact symplectic toric manifolds},
  \href{http://dx.doi.org/10.3842/SIGMA.2015.055}{SIGMA Symmetry Integrability
  Geom. Methods Appl. \textbf{11} (2015)}, Paper 055, 37.

\bibitem{MR0364552}
B.~Kostant, \emph{On convexity, the {W}eyl group and the {I}wasawa
  decomposition}, Ann. Sci. \'Ecole Norm. Sup. (4) \textbf{6} (1973), 413--455
  (1974).

\bibitem{MR1338784}
E.~Lerman, \emph{Symplectic cuts},
  \href{http://dx.doi.org/10.4310/MRL.1995.v2.n3.a2}{Math. Res. Lett.
  \textbf{2} (1995)}, no.~3, 247--258.

\bibitem{MR1734130}
R.~Mazzeo and R.~B. Melrose, \emph{Pseudodifferential operators on manifolds
  with fibred boundaries}, Asian J. Math. \textbf{2} (1998), no.~4, 833--866,
  \href{http://arxiv.org/abs/arXiv:math/9812120v1}{{\tt arXiv:math/9812120v1}}.

\bibitem{MR1707526}
J.~D. McCarthy, \emph{On the asphericity of a symplectic {$M^3\times S^1$}},
  \href{http://dx.doi.org/10.1090/S0002-9939-00-05571-4}{Proc. Amer. Math. Soc.
  \textbf{129} (2001)}, no.~1, 257--264.

\bibitem{Meinrenken1}
E.~Meinrenken, \emph{Symplectic geometry lecture notes},
  \url{http://www.math.toronto.edu/mein/teaching/sympl.pdf}.

\bibitem{MR2102998}
G.~Mikhalkin, \href{http://dx.doi.org/10.1007/0-306-48658-X_6}{\emph{Amoebas of
  algebraic varieties and tropical geometry}}, Different faces of geometry,
  Int. Math. Ser. (N. Y.), vol.~3, Kluwer/Plenum, New York, 2004, pp.~257--300.

\bibitem{MR2511632}
S.~Payne, \emph{Analytification is the limit of all tropicalizations},
  \href{http://dx.doi.org/10.4310/MRL.2009.v16.n3.a13}{Math. Res. Lett.
  \textbf{16} (2009)}, no.~3, 543--556.

\bibitem{MR2801777}
{\'A}.~Pelayo and S.~V{\~u}~Ng{\d{o}}c, \emph{Symplectic theory of completely
  integrable {H}amiltonian systems},
  \href{http://dx.doi.org/10.1090/S0273-0979-2011-01338-6}{Bull. Amer. Math.
  Soc. (N.S.) \textbf{48} (2011)}, no.~3, 409--455.

\bibitem{Pym:2015fk}
B.~Pym, \emph{Elliptic singularities on log symplectic manifolds and
  Feigin--Odesskii Poisson brackets},
  \href{http://arxiv.org/abs/1507.05668}{{\tt 1507.05668}}.

\bibitem{MR586450}
K.~Saito, \emph{Theory of logarithmic differential forms and logarithmic vector
  fields}, J. Fac. Sci. Univ. Tokyo Sect. IA Math. \textbf{27} (1980), no.~2,
  265--291.

\bibitem{Torres01062004}
D.~M. Torres, \emph{Global Classification of Generic Multi-Vector Fields of Top
  Degree}, \href{http://dx.doi.org/10.1112/S0024610703005143}{Journal of the
  London Mathematical Society \textbf{69} (2004)}, no.~3, 751--766.

\end{thebibliography}

\addresses

\end{document}